\documentclass[12pt]{extarticle}
\usepackage[format=plain, font=footnotesize]{caption}
\usepackage{amsmath, amsthm, amssymb, mathtools, hyperref, color}
\usepackage[shortlabels]{enumitem}
\usepackage{graphicx}
\usepackage{tikz}
\usepackage[ruled,linesnumbered, inoutnumbered]{algorithm2e}
\usepackage{xcolor}
\usepackage{listings}
\usepackage{url}
\usetikzlibrary{arrows,calc,fit,matrix,positioning,cd}
\definecolor{light-gray}{RGB}{240,240,240}
\lstdefinelanguage{Julia}%
  {morekeywords={abstract,break,case,catch,const,continue,do,else,elseif,%
      end,export,false,for,function,immutable,import,importall,if,in,%
      macro,module,otherwise,quote,return,switch,true,try,type,typealias,%
      using,while},%
   sensitive=true,%
   morecomment=[l]\#,%
   morecomment=[n]{\#=}{=\#},%
   morestring=[s]{"}{"},%
   morestring=[m]{'}{'},%
}[keywords,comments,strings]%

\oddsidemargin \evensidemargin
\newtheorem{theorem}{Theorem}
\numberwithin{theorem}{section}

\newtheorem{lemma}[theorem]{Lemma}

\theoremstyle{definition}
\newtheorem{definition}[theorem]{Definition}
\newtheorem{remark}[theorem]{Remark}
\newtheorem{example}[theorem]{Example}

\newcommand{\RR}{\mathbb{R}}
\newcommand{\QQ}{\mathbb{Q}}
\newcommand{\ZZ}{\mathbb{Z}}
\newcommand{\PP}{\mathbb{P}}
\newcommand{\CC}{\mathbb{C}}

\date{}

\title{\textbf{Learning Algebraic Varieties from Samples}}

\author{Paul Breiding, Sara Kali\v snik,
 Bernd Sturmfels and Madeleine Weinstein }

\begin{document}
\maketitle
\vspace{-0.75cm}
 \begin{abstract}
 \noindent
 We seek to determine a real algebraic variety from
 a fixed finite subset of points. Existing methods
  are studied and new methods are developed. Our focus lies on
aspects of   topology and algebraic geometry,  such as
 dimension and defining polynomials.
All algorithms are tested on a range of datasets
and made available in a {\tt Julia} package.
      \end{abstract}
\vspace{-0.5cm}
\section{Introduction}

This paper addresses a fundamental problem at the interface of
data science and algebraic geometry.
Given a sample of points $\Omega=\{u^{(1)}, u^{(2)}, \ldots, u^{(m)}\}$
from an unknown variety $V$ in~$\RR^n$, our task is to learn
as much information about $V$  as possible.  No assumptions
on the variety~$V$, the sampling, or the distribution on $V$ are made.
   There can be noise due to rounding, so the points $u^{(i)}$ do not necessarily lie exactly on the variety from which they have been sampled. The variety $V$ is allowed to be singular or reducible. We also consider the
case where $V$ lives
in the projective space $\PP^{n-1}_\RR$.
 We are interested in  questions such~as:

\medskip
\begin{minipage}{0.55\textwidth}
\begin{enumerate}
\item What is the dimension of $V$?
\item Which polynomials vanish on $V$?
\item What is the degree of $V$?
\item What are the irreducible components of $V$?
\item What are the homology groups of $V$?
\end{enumerate}
\end{minipage}
\hspace{0.05\textwidth}
\begin{minipage}{0.35\textwidth}
\begin{center}
 \includegraphics[width=0.4\textwidth]{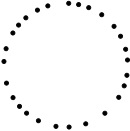}
 \captionof{figure}{Sample of 27 points \\ from an unknown plane curve.}
 \label{fig:circle}
\end{center}
\end{minipage}
\medskip

\noindent
Let us consider these five questions for the dataset with $m=27$ and $n=2$
shown in Figure~\ref{fig:circle}.
Here the answers are easy to see,
but what to do if $n \geq 4$ and no picture is available?
\begin{enumerate}
\item The dimension of the unknown variety $V$ is one.
\item The ideal of $V$ is generated by one polynomial of the form
$(x-\alpha)^2 + (y-\beta)^2 - \gamma$.
\item The degree of $V$ is two. A generic line meets $V$ in two (possibly complex) points.
\item The circle $V$ is irreducible because it admits a  parametrization by rational functions.
\item The homology groups  are $\,H_0(V,\ZZ) = H_1(V,\ZZ) =  \ZZ^1\,$ and
$\,H_i(V,\ZZ) = 0 $ for $i \geq 2$.
\end{enumerate}

There is a considerable body of literature on such questions in
statistics and computer science. The general
context is known as {\em manifold learning}.
One often assumes that $V$ is smooth, i.e.~a manifold, in order
to apply local methods based on approximation by tangent spaces.
Learning the true nature of the manifold $V$ is not a concern for most authors.
Their principal
aim is {\em dimensionality reduction}, and $V$ only serves
in an auxiliary role. Manifolds act as a scaffolding to frame question 1.
This makes sense when  the parameters  $m$ and $n$  are  large.
Nevertheless, the existing literature often draws its inspiration from
figures in $3$-space with many well-spaced sample points.
For instance, the textbook by Lee and Verleysen~\cite{LV}
employs the ``Swiss roll'' and the ``open box'' for its running examples
(cf.~\cite[\S 1.5]{LV}).

 One notable exception is the work by Ma {\it et al.} \cite{derksen}.
 Their {\em Generalized Principal Component Analysis}
 solves  problems 1-4 under the assumption that $V$ is a finite
 union of linear subspaces.
 Question 5 falls under the umbrella
 of {\em topological data analysis} (TDA). Foundational work by
 Niyogi,  Smale and Weinberger \cite{NSW} concerns the
 number $m$ of samples needed to compute the
 homology groups of $V$, provided $V$ is smooth and
 its {\em reach} is known.

The perspective of this paper is that of {\em computational
algebraic geometry}. We care deeply about the unknown
variety $V$. Our motivation is the riddle: {\em what is $V$?}
For instance, we may be given $m=800$ samples in $\RR^9$,
drawn secretly from the group ${\rm SO}(3)$ of $3 {\times} 3$ rotation matrices.
Our goal is to learn the true dimension, which is three,
to find the $20$ quadratic polynomials that vanish on $V$, and
to conclude with the guess that $V$ equals ${\rm SO}(3)$.

Our article is organized as follows.
Section \ref{sec:vardata}  presents basics of algebraic geometry
from a data perspective.
Building on \cite{CLO},
we explain some relevant concepts and
offer a catalogue of varieties $V$ frequently seen in applications.
This includes our three running examples: the Trott curve,
the rotation group ${\rm SO}(3)$, and varieties of
low rank matrices.

Section \ref{sec:dimension} addresses the
problem of estimating the dimension of $V$
from the sample $\Omega$. We study nonlinear PCA, box counting dimension, persistent homology curve dimension, correlation dimension
and the methods of Levina-Bickel \cite{LB} and Diaz-Quiroz-Velasco \cite{DQV}. Each of these notions depends on
a parameter $\epsilon$ between $0$ and $1$. This determines the scale from local to global at which we consider $\Omega$. Our empirical dimensions are functions of $\epsilon$.
We aggregate their graphs in the {\em dimension diagram} of $\Omega$,
as seen in Figure \ref{DD}.

Section \ref{sec:homology}
 links algebraic geometry to topological data analysis. To
learn homological information about $V$  from $\Omega$, one wishes to know the
{\em reach} of the variety $V$. This algebraic number is used
to assess the quality of a sample \cite{AKCMRW, NSW}.
We propose a variant of persistent homology that incorporates information about
the tangent spaces of $V$ at points in $\Omega$.

A key feature of our setting is the existence of
polynomials that vanish on the model~$V$, extracted from polynomials
that vanish on the sample $\Omega$. Linear polynomials are found by
Principal Component Analysis (PCA).  However, many relevant varieties $V$
are defined by quadratic or cubic equations. Section
\ref{sec:finding_equations} concerns the
computation of these polynomials.

Section \ref{sec:using}
utilizes the polynomials found in Section \ref{sec:finding_equations}.
These cut out a variety~$V'$ that contains $V$. We do not know whether
$V' = V$ holds, but we would like to test this and certify it,
using both numerical and symbolic algorithms.
The geography of $\Omega$ inside $V'$ is studied by computing
dimension, degree, irreducible decomposition,
real degree, and volume.

Section \ref{sec:experiments_and_data} introduces our software package
 \texttt{LearningAlgebraicVarieties}. This is written in
 \texttt{Julia} \cite{bezanson2017julia}, and implements all algorithms described in
 this paper. It is  available~at
 $$ \hbox{\url{https://github.com/PBrdng/LearningAlgebraicVarieties.git}}. $$
 To compute persistent homology, we use
 Henselman's package~\texttt{Eirene}~\cite{Eirene}. For
  numerical algebraic geometry we use \texttt{Bertini} \cite{BertiniBook} and \texttt{HomotopyContinuation.jl}~\cite{BT}.
We conclude with a detailed case study for the dataset in
\cite[\S 6.3]{JavaplexTutorial}. Here, $\Omega$
consists of $6040$ points in~$\RR^{24}$, representing
conformations of the molecule cyclo-octane $C_8 H_{16}$,
shown in Figure \ref{cyclo_pic}.

Due to space limitations, many important aspects of learning varieties from samples
are not addressed in this article. One is the issue of noise. Clearly, already the slightest noise in one of the points in Figure~\ref{fig:circle} will let no equation of the form $(x-\alpha)^2 + (y-\beta)^2 - \gamma$ vanish on $\Omega$. But some will \emph{almost} vanish, and these are the equations we are looking for. Based on our experiments, the methods we present for answering questions 1-5 can handle data that is approximate to some extent. However, we leave a qualitative stability analysis for future work. We also assume that there are no outliers in our data. Another aspect of learning varieties is optimization.
 We might be interested in minimizing a polynomial function $f$
 over the unknown variety $V$ by only looking at the samples in $\Omega$.
 This problem was studied by Cifuentes and Parrilo in~\cite{CP},
 using the sum of squares (SOS) paradigm \cite{BPT}.

\section{Varieties and Data}\label{sec:vardata}

The mathematics of data science is concerned with finding
 low-dimensional needles in  high-dimensional haystacks. The needle is the
 model which harbors the actual data, whereas the haystack is some ambient space.
The paradigms of models are the $d$-dimensional linear subspaces $V$ of
$\RR^n$, where $d$ is small and $n$ is large. Most of the points in $\RR^n$ are very far from
any sample $\Omega$ one might ever draw from $V$, even in the presence of noise and outliers.

The data scientist seeks to learn the unknown model $V$ from the
sample $\Omega$ that is available. If $V$ is suspected to be a
linear space, then she uses linear algebra. The first tool that comes to mind is
Principal Component Analysis (PCA). Numerical algorithms for linear algebra
are well-developed and fast. They are at the heart of scientific computing
and its numerous applications. However, many models $V$ occurring in
science and engineering are not linear spaces. Attempts to
replace $V$ with a linear approximation are likely to fail.

This is the point where new mathematics comes in.
Many branches of mathematics can help with
the needles of data science. One can think of $V$ as a topological space, a differential manifold,
a metric space, a Lie group, a hypergraph, a category, a semi-algebraic set,
and lots of other things. All of these structures
are useful in representing and analyzing models.

In this article we  focus on the constraints
that describe $V$ inside the ambient  $\RR^n$ (or~$\mathbb{P}_\mathbb{R}^{n-1}$).
The paradigm says that these are linear equations,
revealed numerically by feeding $\Omega$ to~PCA.
But, if the constraints are not all linear, then
we look for equations of higher degree.

\subsection{Algebraic Geometry Basics}

Our models $V$ are algebraic varieties over the field $\RR$ of real numbers.
A {\em variety} is the set of common zeros of a system of
polynomials in $n$ variables. A priori,  a variety lives in
{\em Euclidean space} $\RR^n$.
In many applications two points are identified if they
agree up to scaling. In such cases, one replaces $\RR^n$ with
the {\em real projective space} $\PP^{n-1}_\RR$, whose points
are lines through the origin in $\RR^n$. The resulting
 model $V$ is a {\em real projective variety},
defined by homogeneous polynomials
in $n$ unknowns. In this article, we use the term {\em variety}
to mean any zero set of polynomials in
$\RR^n$ or $\PP^{n-1}_\RR$.
The following three varieties serve as our running examples.

\begin{example}[Trott Curve] \label{ex:trott}
The Trott curve is the plane curve of degree four defined by
 \begin{equation}
 \label{eq:trotteqn} 12^2(x^4+y^4)\,-\,15^2(x^2+y^2)\,+\,350x^2y^2\,+\,81\,\,\,=\,\,\,0.
\end{equation}
 This curve is compact in $\RR^2$  and
 has four connected components (see Figure \ref{fig:trott2}).
 The equation of the corresponding projective curve is obtained
 by homogenizing the polynomial (\ref{eq:trotteqn}).
 The curve is nonsingular.
The Trott curve is quite special because all of its bitangent lines
are all fully real. Pl\"ucker showed in 1839 that every plane quartic
has $28$ complex bitangents,  Zeuthen argued in 1873
that the number of real bitangents is $28$, $16$, $8$ or $4$;
see \cite[Table~1]{PSV}.

\end{example}

\begin{example}[Rotation Matrices] \label{ex:rotation}
The group ${\rm SO}(3)$ consists of all
$3 {\times} 3$-matrices $X= (x_{ij})$ with
${\rm det}(X) = 1$ and  $X^T X  = {\rm Id}_3$. The last constraint
translates into $9$ quadratic equations:
{\footnotesize
\begin{equation}
\label{eq:nineSO3}
 \begin{matrix} x_{11}^2+x_{21}^2+x_{31}^2-1 & &
x_{11} x_{12}+x_{21} x_{22}+x_{31} x_{32}  && x_{11} x_{13}+x_{21} x_{23}+x_{31}x_{33} \\
     x_{11} x_{12} +x_{21} x_{22} + x_{31} x_{32} &&
      x_{12}^2+x_{22}^2+x_{32}^2-1 && x_{12} x_{13} + x_{22} x_{23} + x_{32} x_{33} \\
     x_{11} x_{13} +x_{21} x_{23} + x_{31} x_{33}   &&
     x_{12} x_{13} +x_{22} x_{23} +x_{32} x_{33} & & x_{13}^2 + x_{23}^2 + x_{33} ^2-1
     \end{matrix}
\end{equation}}
These quadrics say that $X$ is
an orthogonal matrix. Adding the cubic  ${\rm det}(X)-1$
gives $10$ polynomials that
define ${\rm SO}(3)$ as a variety in $\RR^9$. Their ideal $I$ is prime. In total, there are $20$ linearly independent
quadrics in $I$: the nine listed in  (\ref{eq:nineSO3}),
two from the diagonal of  $XX^T-{\rm Id}_3$, and nine
that express  the right-hand rule for orientation, like
  $ x_{22} x_{33} - x_{23} x_{32} - x_{11}$.
   \end{example}

\begin{example}[Low Rank Matrices] \label{ex:rankone}
Consider the set of $m \times n$-matrices of rank $\leq r$.
This is the zero set of $\binom{m}{r+1} \binom{n}{r+1}$
polynomials, namely the $(r+1)$-minors.
These equations are homogeneous of degree $r+1$. Hence this variety  lives
naturally in the projective space $\PP^{mn-1}_\RR$.
\end{example}

A variety $V$ is {\em irreducible} if it is not a union of two proper subvarieties.
The above varieties are irreducible.
A sufficient condition for a variety to be irreducible is that it has
a parametrization by rational functions. This holds
in Example \ref{ex:rankone} where $V$ consists of
 the matrices  $U_1^T U_2$ where $U_1$ and $U_2$ have $r$ rows.
It also holds for the rotation matrices
\begin{equation}
{\footnotesize
\label{eq:SO3para}
 X \,\,=\,\, \frac{1}{1{-}a^2{-}b^2{-}c^2{-}d^2} \begin{pmatrix}
1{-}2b^{2}{-}2c^{2} & 2ab-2cd & 2ac+2bd \\
2ab+2cd & 1{-}2a^{2}{-}2c^{2} & 2bc-2ad \\
2ac-2bd & 2bc+2ad & 1{-}2a^{2}{-}2b^{2} \end{pmatrix}.}
\end{equation}
 However, smooth quartic
curves in $\PP^2_\RR$ admit no such rational parametrization.

The two most basic invariants of a variety $V$ are its
{\em dimension} and its {\em degree}.
The former is the length $d$ of the
 longest proper chain of irreducible varieties
$V_1 \subset V_2 \subset \cdots \subset V_d \subset V$.
A general system of $d$ linear equations has
a finite number of solutions on $V$. That number is well-defined if we work over $\CC$. It
is the degree of $V$, denoted ${\rm deg}(V)$.
The Trott curve has dimension $1$ and degree $4$.
The group ${\rm SO}(3)$  has dimension~$3$ and degree $8$.
In Example \ref{ex:rankone}, if $m=3,n=4$ and $r=2$, then the projective variety has dimension $9$ and degree $6$.

There are several alternative definitions of dimension and degree in algebraic geometry.
For instance, they are read off from the Hilbert polynomial,
which can be computed by way of Gr\"obner bases.
We refer to Chapter~9, titled {\em Dimension Theory}, in the textbook \cite{CLO}.

A variety that admits a rational parametrization is called {\em unirational}.
Smooth plane curves of degree $\geq 3$ are not unirational.
However, the varieties $V$ that arise in applications are often unirational.
The reason is that $V$ often models a generative process.
This happens in statistics, where $V$ represents some kind of
(conditional) independence structure. Examples include graphical models,
hidden Markov models and phylogenetic models.

If $V$ is a unirational variety with given rational parametrization,
then it is easy to create a finite subset $\Omega$  of $V$. One
selects parameter values at random and plugs these into the
parametrization.
For instance, one creates
rank one matrices by simply multiplying a random column vector
with a random row vector. A naive approach to sampling
from the rotation group ${\rm SO}(3)$
is plugging four random real numbers $a,b,c,d$  into the
 parametrization (\ref{eq:SO3para}).
Another  method for sampling from ${\rm SO}(3)$
will be discussed in Section  \ref{sec:experiments_and_data}.

Given a dataset $\Omega \subset \RR^n$ that comes from an
applied context, it is reasonable to surmise that the underlying
unknown variety $V$ admits a rational parametrization.
However, from the vantage point of a pure geometer,
such unirational varieties are rare. To sample from a general variety $V$,
we start from its defining equations, and we solve
${\rm dim}(V)$ many linear equations on $V$.
The algebraic complexity of carrying this out is measured by ${\rm deg}(V)$.
See Dufresne {\it et al.}~\cite{DEHH} for
recent work on sampling  by way of numerical algebraic geometry.

\begin{example} \label{ex:trottsample}
One might sample from the Trott curve $V$ in Example~\ref{ex:trott}
by intersecting it with a random line. Algebraically, one solves
$\,{\rm dim}(V)=1\,$ linear equation on the curve. That line intersects $V$ in
$\,{\rm deg}(V) = 4\,$ points. Computing the intersection points can be done numerically, but also symbolically by using Cardano's formula for the quartic.
In either case, the coordinates computed by these methods may be complex numbers.
Such points are simply discarded if real samples are desired.
This can be a rather wasteful process.

At this point, optimization and real algebraic geometry enter the scene.
Suppose that upper and lower bounds are known for
the values of a linear function $\ell$ on $V$. In that case, the equations
to solve have the form $\ell(x) = \alpha$, where $\alpha$ is chosen  between these  bounds.

For the Trott curve, we might know that no real points exist unless $|x| \leq 1$.
We choose $x$ at random between $-1$ and $+1$, plug it into the
equation (\ref{eq:trotteqn}), and then solve the resulting quartic in $y$.
The solutions $y$ thus obtained are likely to be real, thus giving us lots
of real samples on the curve. Of course, for arbitrary real varieties,
it is a hard problem to identify a priori constraints that play the role
of $|x| \leq 1$. However, recent advances in polynomial optimization,
notably in sum-of-squares programming \cite{BPT}, should be quite~helpful.
\end{example}

At this point, let us recap and focus on a concrete instance of the riddles we seek to solve.

\begin{example} \label{ex:toblerone}
\rm
Let $n=6$, $m=40$ and consider the following forty sample points in $\RR^6$:
$$ \begin{tiny} \begin{matrix}
(0, -2, 6, 0, -1, 12) &
(-4, 5, -15, -12, -5, 15) \! &
(-4, 2, -3, 2, 6, -1) \! &
(0, 0, -1, -6, 0, 4)  \\
\! \! (12, 3, -8, 8, -12, 2) &
\!\!\!\! (20, 24, -30, -25, 24, -30) \! \!\!&
(9, 3, 5, 3, 15, 1) &
\! \! (12, 9, -25, 20, -15, 15) \\
(0, -10, -12, 0, 8, 15) &
(15, -6, -4, 5, -12, -2) &
(3, 2, 6, 6, 3, 4) &
(12, -8, 9, 9, 12, -6) \\
(2, -10, 15, -5, -6, 25) &
(5, -5, 0, -3, 0, 3) &
(-12, 18, 6, -8, 9, 12) &
\! (12, 10, -12, -18, 8, -15) \\
(1, 0, -4, -2, 2, 0) &
(4, -5, 0, 0, -3, 0) &
(12, -2, 1, 6, 2, -1) &
(-5, 0, -2, 5, 2, 0) \\
(3, -2, -8, -6, 4, 4) &
(-3, -1, -9, -9, -3, -3) &
(0, 1, -2, 0, 1, -2) &
\!\! (5, 6, 8, 10, 4, 12) \\
(2, 0, -1, -1, 2, 0) &
(12, -9, -1, 4, -3, -3) &
\!\!\! (5, -6, 16, -20, -4, 24) \!\! &
(0, 0, 1, -3, 0, 1) \\
\! (15, -10, -12, 12, -15, -8) \!\!\! \! &
(15, -5, 6, 6, 15, -2) &
(-2, 1, 6, -12, 1, 6) &
(3, 2, 0, 0, -2, 0) \\
(24, -20, -6, -18, 8, 15) &
(-3, 3, -1, -3, -1, 3) &
(-10, 0, 6, -12, 5, 0) &
(2, -2, 10, 5, 4, -5) \\
(4, -6, 1, -2, -2, 3) &
(3, -5, -6, 3, -6, -5) &
(0, 0, -2, 3, 0, 1) &
\! (-6, -4, -30, 15, 12, 10)
\end{matrix}
\end{tiny}
$$
Where do these samples come from?
Do the zero entries or the sign patterns offer any clue?

\smallskip

To reveal the answer we label the coordinates as
$(x_{22},x_{21},x_{13},x_{12},x_{23},x_{11})$. The relations
$$\,x_{11} x_{22} - x_{12} x_{21} \,=\,
    x_{11} x_{23} - x_{13} x_{21} \,=\,
    x_{12} x_{23} - x_{22} x_{13} \,=\, 0 $$
hold for all $40$ data points. Hence $V$ is the variety
of $2 \times 3$-matrices~$(x_{ij})$ of rank $\leq 1$.
Following Example \ref{ex:rankone}, we view this
as a projective variety in $\PP^5_\RR$.
In that ambient projective space, the determinantal
 variety $V$ is a manifold of dimension $3$ and degree $3$.
Note that $V$ is homeomorphic to $\PP^1_\RR \times \PP^2_\RR$,
so we can write its homology groups using the K\"unneth formula.
\end{example}

In data analysis, proximity between sample points
plays a crucial role. There are many ways
to measure distances. In this paper we restrict
ourselves to two metrics. For data in~$\RR^n$ we use the Euclidean metric, which is induced
by the standard inner product ${\langle u,v \rangle = \sum_{i=1}^n u_i v_i }$.
For data in $\PP^{n-1}_\RR$
we use the Fubini-Study metric. Points $u$ and $v$
in~$\PP^{n-1}_\RR$ are represented by their homogeneous
coordinate vectors. The {\em Fubini-Study distance}
from $u$ to $v$  is the angle between the lines spanned
by representative vectors $u$ and $v$ in $\RR^n$:
\begin{equation}
\label{eq:FS}
 \mathrm{dist}_\mathrm{FS}(u, v) \,\, = \,\,
\arccos \frac{\vert \langle u,v \rangle\vert}{\Vert u\Vert \Vert v\Vert}.
\end{equation}
This formula defines the unique  Riemannian metric on $\PP^{n-1}_\RR$
that is  orthogonally invariant.

\subsection{A Variety of Varieties}
\label{subsec:variety}

In what follows we present some ``model organisms''
seen in applied algebraic geometry.
Familiarity with a repertoire of interesting varieties
is an essential prerequisite for those who are serious about
learning algebraic structure from the datasets $\Omega$ they might encounter.

\medskip \noindent {\bf Rank Constraints.}
Consider $m \times n$-matrices with linear entries having rank $\leq r$.
We saw the  $r=1$ case in Example \ref{ex:rankone}.
A {\em rank variety} is the set of all tensors of fixed size and rank
that satisfy some linear constraints.
The constraints often take the simple form that two entries are equal.
This includes symmetric matrices, Hankel matrices, Toeplitz matrices,
Sylvester matrices, etc. Many classes of structured matrices
generalize naturally to tensors.

\begin{example} \label{ex:pfaffians} \rm
Let $n = \binom{s}{2}$ and identify $\RR^n$ with the space
of skew-symmetric $s \times s$-matrices $P= (p_{ij}) $.
These  satisfy $P^T = - P$.
Let $V$ be the variety of  rank $2$ matrices $P$ in
$\PP^{n-1}_\RR$. A parametric representation is given by
$p_{ij} = a_i b_j - a_j b_i$, so the $p_{ij}$ are the $2 \times 2$-minors of
a $2 \times s$-matrix. The ideal of $V$ is generated by the
{\em $4 \times 4$ pfaffians} $\, p_{ij} p_{kl} - p_{ik} p_{jl} + p_{il} p_{jk}$.
These  $\binom{s}{4}$ quadrics
are also known as the {\em Pl\"ucker relations}, and
 $V$ is the {\em Grassmannian} of $2$-dimensional
linear subspaces in $\RR^s$. The $r$-{\em secants} of $V$ are represented by
 the variety of skew-symmetric matrices
of rank $\leq 2r$. Its equations are the $(2r{+}2) \times (2r{+}2)$ pfaffians of~$P$.
We refer to \cite[Lectures 6 and 9]{Har} for an introduction to these classical varieties.
\end{example}

\begin{example}
\label{ex:sym3333}
The space of $3 \times 3 \times 3 \times 3$ tensors $(x_{ijkl})_{1 \leq i,j,k,l \leq 3}$
has dimension $81$. Suppose we sample from
its subspace of symmetric tensors $m = (m_{rst})_{0 \leq r \leq s \leq t \leq 3}$.
This has dimension   $n=20$.
We use the convention
$m_{rst} = x_{ijkl}$ where $r$ is the number of indices $1$ in $(i,j,k,l)$, $s$ is the number of indices $2$, and $t$ is the number of indices $3$.
This identifies tensors $m$ with cubic polynomials $m = \sum_{i+j+k \leq 3} m_{ijk} x^i y^j z^k$,
and hence with cubic surfaces in $3$-space.
Fix $r \in \{1,2,3\}$ and take $V$ to be the variety of tensors $m$
of rank $\leq r$. The equations that define the tensor rank variety $V$ are the
$(r+1) \times (r+1)$-minors of the  $4 \times 10$ {\em Hankel matrix}
$$ {\footnotesize
\begin{bmatrix}
                                                                       \, m_{000}\, & \, m_{100} & m_{010} &  m_{001} \,& \,m_{200} &
 m_{110} & m_{101} & m_{020} & m_{011} & m_{002} \, \\
                                                                         \, m_{100} \,&\, m_{200} &  m_{110}  & m_{101} \,&\, m_{300} &
 m_{210} & m_{201} & m_{120} & m_{111} & m_{102}  \, \\
                                                                          \, m_{010} \,&\, m_{110} & m_{020} & m_{011}  \,&\, m_{210} &
m_{120} & m_{111} &  m_{030} & m_{021} &  m_{012} \, \\
                                                                          \, m_{001} \,&\, m_{101} & m_{011} &  m_{002} \,&\, m_{201} &
m_{111} & m_{102} & m_{021} & m_{012} & m_{003} \, \\
\end{bmatrix}.}
$$
See Landsberg's book \cite{Lan} for an introduction to the geometry of tensors and their rank.
\end{example}

\begin{example}
\label{ex:8schoenberg}
In {\em distance geometry}, one encodes finite
metric spaces with $p$ points in the {\em Sch\"onberg matrix}
$\,D \,=\, \bigl(d_{ip}+d_{jp}-d_{ij} \bigr)\,$ where
 $d_{ij}$ is the squared distance between points~$i$ and~$j$.
 The symmetric $(p{-}1) \times (p{-}1)$ matrix $D$ is positive
semidefinite if and only if the metric space is Euclidean,
and its embedding dimension is the rank $r$ of $D$.
See \cite[\S 6.2.1]{DL} for a textbook introduction and derivation of Sch\"onberg's esults.
Hence the rank varieties of the Sch\"onberg matrix $D$ encode the
finite Euclidean metric spaces with $p$ points.
A prominent dataset corresponding to the case $p=8$ and $r=3$ will
be studied in Section \ref{sec:experiments_and_data}.
\end{example}

Matrices and tensors with rank constraints are ubiquitous in data science.
Make sure to search for such low rank structures when facing
vectorized samples, as in~Example \ref{ex:toblerone}.

\medskip \noindent  {\bf Hypersurfaces}.
The most basic varieties are defined by just one polynomial.
When given a sample $\Omega$, one might begin
by asking for  hypersurfaces that contain $\Omega$
and that are especially nice, simple and informative.
Here are some examples of special structures
worth looking~for.

\begin{example}
For $s=6, r=2$  in  Example \ref{ex:pfaffians}, $V$ is the hypersurface of the {\em $6 \times 6$-pfaffian}:
\begin{equation}
\label{eq:6pfaffian} {\footnotesize
  \begin{matrix} \,\, \,\, p_{16} p_{25} p_{34} - p_{15} p_{26} p_{34} - p_{16} p_{24} p_{35}
+ p_{14} p_{26} p_{35} + p_{15} p_{24} p_{36} \\
- p_{14} p_{25} p_{36}
+ p_{16} p_{23} p_{45} - p_{13} p_{26} p_{45} + p_{12} p_{36} p_{45}
- p_{15} p_{23} p_{46}  \\ + p_{13} p_{25} p_{46} - p_{12} p_{35} p_{46}
+ p_{14} p_{23} p_{56} - p_{13} p_{24} p_{56} + p_{12} p_{34} p_{56}.
\end{matrix}}
\end{equation}
The $15$ monomials correspond to the matchings of the complete graph with six vertices.
\end{example}

\begin{example}
The {\em hyperdeterminant} of format $2 \times 2 \times 2 $ is
a polynomial of degree four in $n=8$ unknowns, namely
the entries of a $2 \times 2 \times 2$-tensor
$X = (x_{ijk})$. Its expansion equals
$$ {\footnotesize \begin{matrix}
    x_{110}^2 x_{001}^2 {+} x_{100}^2 x_{011}^2 {+}x_{010}^2 x_{101}^2 {+}
    x_{000}^2 x_{111}^2   + 4 x_{000} x_{110} x_{011} x_{101}  {+}4 x_{010} x_{100} x_{001}  x_{111}
    -2x_{100} x_{110} x_{001} x_{011}  \\ -
    2 x_{010} x_{110} x_{001} x_{101}-2x_{010} x_{100} x_{011} x_{101}
    {-}2 x_{000} x_{110} x_{001} x_{111}
    {-}2 x_{000} x_{100} x_{011}  x_{111}
    {-} 2 x_{000} x_{010} x_{101} x_{111}.
\end{matrix}
} $$
This hypersurface is rational and it admits several nice parametrizations,
useful for sampling points. For instance, up to scaling, we can take the eight principal
minors of a symmetric $3 \times 3$-matrix, with
$x_{000} = 1$ as the $0 \times 0$-minor,
$x_{100},x_{010},x_{001}$ for the $1 \times 1 $-minors (i.e. diagonal entries),
$x_{110},x_{101}, x_{011}$ for the $2 \times 2$-minors,
and $x_{111} $ for the $3 \times 3$-determinant.
\end{example}

\begin{example}
Let $n=10$, with coordinates for $\RR^{10}$ given by the
off-diagonal entries of a symmetric $5 \times 5$-matrix $(x_{ij})$.
There is a unique quintic polynomial in these variables
that vanishes on symmetric $5 \times 5$-matrices of rank $\leq 2$.
This polynomial, known as the {\em pentad}, plays a historical role
in the statistical theory of {\em factor analysis} \cite[Example 4.2.8]{DSS}.
It equals
$$ {\footnotesize \begin{matrix}
  x_{14} x_{15} x_{23} x_{25} x_{34}
- x_{13} x_{15} x_{24} x_{25} x_{34}
- x_{14} x_{15} x_{23} x_{24} x_{35}
+ x_{13} x_{14} x_{24} x_{25} x_{35} \\
+ x_{12} x_{15} x_{24} x_{34} x_{35}
- x_{12} x_{14} x_{25} x_{34} x_{35}
+ x_{13} x_{15} x_{23} x_{24} x_{45}
- x_{13} x_{14} x_{23} x_{25} x_{45} \\
- x_{12} x_{15} x_{23} x_{34} x_{45}
+ x_{12} x_{13} x_{25} x_{34} x_{45}
+ x_{12} x_{14} x_{23} x_{35} x_{45}
- x_{12} x_{13} x_{24} x_{35} x_{45}.
\end{matrix}
} $$
We  can sample from the pentad using  the parametrization
$\,x_{ij} =  a_i b_j + c_i d_j \,$
for $ 1 \leq i < j \leq 5$.
\end{example}

\begin{example}
The determinant of the $ (p{-}1) \times (p{-}1) $ matrix in Example~\ref{ex:8schoenberg}
equals the squared volume of the simplex spanned by
$p$ points in $\RR^{p-1}$. If $p=3$ then we get Heron's formula
for the area of a triangle in terms of its side lengths.
 The hypersurface in $\RR^{\binom{p}{2}}$ defined by this polynomial represents
configurations of $p$ points in $\RR^{p-1}$ that are degenerate.
\end{example}

One problem with interesting hypersurfaces is that they
often have a very high degree and it would be impossible
to find that equation by our methods in Section \ref{sec:finding_equations}.
For instance, the {\em L\"uroth hypersurface} \cite{BLRS} in the space of
ternary quartics has degree $54$, and the {\em restricted Boltzmann machine} \cite{CMS}
on four binary random variables has degree $110$.
These hypersurfaces are easy to sample from,
but there is little hope to learn their equations from those samples.

\bigskip

\noindent {\bf Secret Linear Spaces.}
This refers to varieties that become linear spaces
after a simple change of coordinates.
Linear spaces $V$ are easy to recognize from samples $\Omega$ using PCA.

{\em Toric varieties} become linear spaces after taking logarithms, so they can be
learned by taking the coordinatewise logarithm of the sample points.
Formally, a toric variety is the image of a monomial map. Equivalently, it is
 an irreducible variety defined by binomials.

\begin{example} \rm
Let $n=6, m=40$ and consider the following dataset in $\RR^6$:
$$
\begin{tiny} \begin{matrix}
(91, 130, 169, 70, 91, 130)\!\! &  (4, 2, 1, 8, 4, 2) &
(6, 33, 36, 11, 12, 66) &  (24, 20, 44, 30, 66, 55) \\
(8, 5, 10, 40, 80, 50) &  (11, 11, 22, 2, 4, 4) &
(88, 24, 72, 33, 99, 27) &  (14, 77, 56, 11, 8, 44) \\
(70, 60, 45, 84, 63, 54) &  (143, 13, 78, 11, 66, 6) &
\!\! (182, 91, 156, 98, 168, 84) \! &  (21, 98, 91, 42, 39, 182) \\
(5, 12, 3, 20, 5, 12) &  (80, 24, 8, 30, 10, 3) &
(3, 5, 5, 15, 15, 25) &  (10, 10, 11, 10, 11, 11) \\
(121, 66, 88, 66, 88, 48) &  (45, 81, 63, 45, 35, 63) &
(48, 52, 12, 156, 36, 39) &  (45, 50, 60, 45, 54, 60) \\
(143, 52, 117, 44, 99, 36) &  (56, 63, 7, 72, 8, 9) &
(10, 55, 20, 11, 4, 22) & (91, 56, 7, 104, 13, 8) \\
(24, 6, 42, 4, 28, 7) & (18, 10, 18, 45, 81, 45) &
(36, 27, 117, 12, 52, 39) & (3, 2, 2, 3, 3, 2) \\
(40, 10, 35, 8, 28, 7) & \!\! (22, 10, 26, 55, 143, 65) &
(132, 36, 60, 33, 55, 15) & \!\! (98, 154, 154, 77, 77, 121) \\
(55, 20, 55, 44, 121, 44) &  (24, 30, 39, 40, 52, 65) & \!
(22, 22, 28, 121, 154, 154) &  (6, 3, 6, 4, 8, 4) \\
(77, 99, 44, 63, 28, 36) &  (30, 20, 90, 6, 27, 18) &
(1, 5, 2, 5, 2, 10) &  (26, 8, 28, 26, 91, 28) \\
\end{matrix}
\end{tiny}
$$
Replace each of these forty vectors by its coordinate-wise logarithm.
Applying PCA to the resulting vectors, we learn that our sample
comes from a $4$-dimensional
subspace of $\RR^6$. This is the row space of a $4 \times 6$-matrix
whose columns are the vertices of a regular octahedron:
$$ {\footnotesize
A \quad = \quad
\begin{pmatrix}
1 & 1 & 1 & 0 & 0 & 0 \\
1 & 0 & 0 & 1 & 1 & 0 \\
0 & 1 & 0 & 1 & 0 & 1 \\
0 & 0 & 1 & 0 & 1 & 1 \\
\end{pmatrix}. }
$$
Our original samples came from the toric variety $X_A$ associated with this matrix.
This means each sample has the form
$(ab,ac,ad,bc,bd,cd)$, where $a,b,c,d$ are positive real numbers.
\end{example}

Toric varieties are important in applications. For instance, in statistics
they correspond to {\em exponential families} for discrete random variables.
Overlap with rank varieties arises for matrices and tensors of rank $1$.
Those smallest rank varieties are known in geometry as
the {\em Segre varieties} (for arbitrary tensors) and
the {\em Veronese varieties} (for symmetric tensors).
These special varieties are toric, so they are represented by an integer matrix $A$ as above.

\begin{example} \label{ex:rls} \rm
Let $n=6$ and take $\Omega$ to be a sample of points of the form
{\footnotesize$$
\bigl(\, (2 a + b)^{-1}, (a + 2 b)^{-1}, (2 a + c)^{-1}, (a + 2 c)^{-1}, (2 b + c)^{-1}, (b + 2 c)^{-1} \,\bigr). $$}
The corresponding variety $V \subset \PP^5_\RR$ is
a {\em reciprocal linear space} $V$; see \cite{KV}.
In projective geometry, such a variety arises as
 the image of a linear space under the classical {\em Cremona transformation}.
From the sample we can learn the variety $V$
by replacing each data point by its coordinate-wise inverse.
Applying PCA to these reciprocalized data, we learn that  $V$
is a surface in $\PP_\RR^5$, cut out by ten cubics like
$\,2 x_3 x_4 x_5-x_3 x_4 x_6-2 x_3 x_5 x_6+x_4 x_5 x_6 $.
\end{example}

\bigskip

\noindent  {\bf Algebraic Statistics and Computer Vision}.
Model selection is a standard task in statistics.
The models considered in algebraic statistics  \cite{DSS}
are typically semi-algebraic sets, and it is customary to
identify them with their Zariski closures, which are algebraic varieties.

\begin{example} \rm
{\em Bayesian networks} are also known as directed graphical models.
The corresponding varieties are parametrized by
 monomial maps from products of simplices.
Here are the equations for a Bayesian network on $4$ binary random variables
 \cite[Example 3.3.11]{DSS}:
$$ {\footnotesize \begin{matrix}
(x_{0000} + x_{0001})(x_{0110} + x_{0111}) - (x_{0010}+x_{0011})(x_{0100}+x_{0101}), \\
(x_{1000} + x_{1001})(x_{1110}+x_{1111}) - (x_{1010} + x_{1011})(x_{1100}+x_{1101}), \\
x_{0000} x_{1001} - x_{0001} x_{1000}, \,
x_{0010} x_{1011} - x_{0011} x_{1010}, \,
x_{0100} x_{1101} - x_{0101} x_{1100}, \,
x_{0110} x_{1111} - x_{0111} x_{1110}.
\end{matrix}
} $$
The coordinates $x_{ijkl}$ represent the probabilities of observing
the $16$ states under this~model.
\end{example}

Computational biology is an excellent source of
statistical models with interesting geometric and
combinatorial properties. These include
hidden variable tree models for phylogenetics,
and  hidden Markov models for gene annotation
and sequence alignment.

In the social sciences and economics, statistical
models for permutations are widely used:

\begin{example} \rm
Let $n=6$ and consider the {\em Plackett-Luce model} for rankings of three items~\cite{SW}.
Each item has a model parameter $\theta_i$, and we write $x_{ijk}$
for the probability of observing the permutation $ijk$.
The model is the surface in $\PP^5_\RR$ given by the parametrization
$$  {\footnotesize
    \begin{matrix}
        x_{123} \,=\,
          \theta_2 \theta_3 (\theta_1 {+}\theta_3 ) (\theta_2 {+}\theta_3 ), &
        x_{132} \,=\,\theta_2 \theta_3 (\theta_1 {+}\theta_2) (\theta_2 {+}\theta_3), &
        x_{213} \,= \, \theta_1 \theta_3 (\theta_1 {+}\theta_3) (\theta_2 {+}\theta_3) ,\\
        x_{231} \, = \, \theta_1 \theta_3 (\theta_1 {+}\theta_2) (\theta_1 {+}\theta_3), &
        x_{312} \, = \, \theta_1 \theta_2 (\theta_1 {+}\theta_2) (\theta_2 {+}\theta_3) ,&
        x_{321} \, = \, \theta_1 \theta_2 (\theta_1 {+}\theta_2) (\theta_1 {+} \theta_3) .
      \end{matrix}
}    $$
    The prime ideal of this model is generated by three quadrics and one cubic:
$$  {\footnotesize
      \begin{matrix}
        x_{123}(x_{321} + x_{231})-x_{213}(x_{132} + x_{312})\,,\,\,
        x_{312}(x_{123} + x_{213})-x_{132}(x_{231} + x_{321}), \\
        x_{231}(x_{132} + x_{312})-x_{321}(x_{123} + x_{213}), \quad\,
        x_{123}x_{231}x_{312}-x_{132}x_{321}x_{213}.
      \end{matrix}
}          $$
\end{example}

When dealing with continuous distributions,
we can represent certain statistical models as varieties in
moment coordinates.  This applies to
Gaussians and their mixtures.

\begin{example} \rm
Consider the projective variety
in $\PP_\RR^6$ given parametrically by $m_0 = 1$ and
$$ {\footnotesize
 \begin{matrix}
m_1 & = & \lambda \mu + (1-\lambda) \nu \\
m_2 & = & \lambda (\mu^2 + \sigma^2) + (1-\lambda) (\nu^2 + \tau^2) \\
m_3 & = & \lambda (\mu^3 + 3 \mu \sigma^2) + (1-\lambda) (\nu^3 + 3 \nu \tau^2) \\
m_4 & = & \lambda (\mu^4 + 6 \mu^2 \sigma^2 + 3 \sigma^4)
          + (1-\lambda) (\nu^4 + 6 \nu^2 \tau^2  + 3 \tau^4) \\
m_5 & = & \lambda (\mu^5 + 10 \mu^3 \sigma^2 + 15 \mu \sigma^4)
          + (1-\lambda) (\nu^5 + 10 \nu^3 \tau^2  + 15 \nu \tau^4) \\
m_6 & = & \lambda (\mu^6 + 15 \mu^4 \sigma^2 + 45 \mu^2 \sigma^4 + 15 \sigma^6)
          + (1-\lambda) (\nu^6 + 15 \nu^4 \tau^2  + 45 \nu^2 \tau^4 + 15 \tau^6).
\end{matrix}
} $$
These are the moments of order $\leq 6$ of  the mixture of two
Gaussian random variables on the line. Here $\mu$ and $\nu$
are the means, $\sigma$ and $\tau$ are the variances, and $\lambda$
is the mixture parameter. It was shown in
\cite[Theorem 1]{AFS} that this is a hypersurface of degree $39$ in $\PP^6$.
For $\mu = 0$ we get the  {\em Gaussian moment surface}
 which is defined by the   $3 \times 3$-minors of the
$3 \times 6$-matrix
$$ {\footnotesize
\begin{pmatrix}
0 & m_0 & 2 m_1 & 3 m_2 & 4 m_3 & 5 m_4 \\
m_0 & m_1 & m_2 & m_3 & m_4 & m_5 \\
m_1 & m_2 & m_3 & m_4 & m_5 & m_6
\end{pmatrix}.
} $$
\end{example}

\begin{example}
Let $n=9$ and fix the space of $3 \times 3$-matrices.
An {\em essential matrix} is the product of
a rotation matrix times a skew-symmetric matrix.
In computer vision, these matrices
represent the relative position of
two calibrated cameras in $3$-space.
Their entries $x_{ij}$ serve as invariant coordinates for pairs of such cameras.
The variety of essential matrices is defined
by ten cubics. These are known as the {\em D\'emazure cubics} \cite[Example 2.2]{KKPS}.

The article \cite{KKPS} studies  camera models
in the presence  of distortion.
For example, the model described in
\cite[Example 2.3]{KKPS} concerns
{\em essential matrices plus one focal length unknown}.
This is the codimension two variety defined by the
$3 \times 3$-minors of the $3 \times 4$-matrix
$$ {\footnotesize
 \begin{pmatrix}
     \,x_{11} &  x_{12}  & x_{13} & \,\,x_{21} x_{31}+x_{22} x_{32} + x_{23} x_{33} \\
     \,x_{21} & x_{22} & x_{23}  &  -x_{11} x_{31}-x_{12} x_{32}-x_{13} x_{33} \\
     \,x_{31}  & x_{32} & x_{33} & 0     \end{pmatrix}.
}      $$
Learning such models is important for
image reconstruction in computer vision.
\end{example}

\section{Estimating the Dimension}\label{sec:dimension}

The first question one asks about a variety $V$ is
 ``What is the dimension?''. In what follows, we discuss methods for estimating
$ {\rm dim}(V)$ from the finite sample $\Omega$, taken from $V$.
We present six dimension estimates. They are motivated and justified by geometric considerations.
For a manifold, dimension is defined in terms of local charts.
This is consistent with the notion of dimension in algebraic geometry
\cite[Chapter 9]{CLO}.
The dimension estimates in this section are based on $\Omega$ alone.
Later sections will address the computation of equations that vanish on $V$. These
can be employed to find upper bounds on ${\rm dim}(V)$; see (\ref{how_to_set_s}).
In what follows, however, we do not have that information.
All we are given is the input $\Omega=\{u^{(1)},\ldots, u^{(m)}\}$.

\subsection{Dimension Diagrams}

There is an extensive literature  (see e.g.~\cite{CAMASTRA2003, CAMASTRA2016})
on computing an
{\em intrinsic dimension} of the sample $\Omega$ from a manifold $V$.
The intrinsic dimension of $\Omega$ is a positive real number that approximates the {\em Hausdorff dimension} of $V$, a quantity that measures the local dimension of a space using the distances between nearby points. It is a priori not clear that the algebraic definition of ${\rm dim}(V)$
agrees with the topological definition of  Hausdorff dimension
that is commonly used  in  manifold learning. However,
this will be true under the following natural hypotheses.
We assume that $V$ is a variety in $\RR^n$ or $\PP^{n-1}_\RR$
such that the set of real points is Zariski dense
in each irreducible component of $V$. If $V$ is irreducible,
then its singular locus ${\rm Sing}(V)$ is a proper subvariety,
so it has measure zero. The regular locus $V \backslash {\rm Sing}(V)$
is a real manifold. Each connected component is a real manifold of dimension
$\,d = {\rm dim}(V)$.

The definitions of intrinsic dimension
 can be grouped into two categories: \emph{local methods} and
 \emph{global methods}~\cite{CAMASTRA2016, JainDubes}. Definitions involving information about sample
 neighborhoods fit into the local category, while those that use the whole
 dataset are called global.

 Instead of making such a strict distinction between local and global, we introduce a parameter $0\leq\epsilon\leq 1$. The idea behind this is that $\epsilon$ should determine the range of information that is used to compute the dimension from the local scale ($\epsilon=0$) to the global scale ($\epsilon=1$).

 To be precise, for each of the dimension estimates, locality is determined by a notion of \emph{distance}:
  the point sample $\Omega$ is a finite metric space. In our context we restrict extrinsic metrics to the sample.
  For samples $\Omega\subset\mathbb{R}^n$
  we work with the {\em scaled Euclidean distance} \begin{equation}\label{Euclidean_scaled}\mathrm{dist}_{\mathbb{R}^n}(u,v) \,\,\,:=\,\,\, \frac{\Vert u - v \Vert}{\max_{x,y\in\Omega}\,\Vert x - y \Vert}.
      \end{equation}
  For samples $\Omega$ taken in projective space $\PP^{n-1}_\RR$ we use the
  {\em scaled Fubini-Study distance}
  \begin{equation}\label{FS_scaled} \,\, \mathrm{dist}_{\PP^{n-1}_\RR}(u,v)
  \,\,:=\,\, \frac{\mathrm{dist}_\mathrm{FS}(u,v)}{\max_{x,y\in\Omega} \mathrm{dist}_\mathrm{FS}(x,y)}.
      \end{equation}
Two points $u^{(i)}$ and $u^{(j)}$ in $\Omega$ are considered {\em $\epsilon$-close}
with respect to the parameter $\epsilon$
if $\mathrm{dist}_{\RR^{n}}(u,v)\leq \epsilon$ or $\mathrm{dist}_{\PP^{n-1}_\RR}(u,v)\leq \epsilon$, respectively. Given $\epsilon$ we divide the sample $\Omega$ into clusters $\Omega_1^\epsilon,\ldots,\Omega_l^\epsilon$, which are defined in terms of $\epsilon$-closeness, and apply the methods to each cluster separately, thus obtaining dimension estimates whose definition of being local depends on $\epsilon$. In particular, for $\epsilon =0$ we consider each sample point individually, while for $\epsilon = 1$ we consider the whole sample. Intermediate values of~$\epsilon$ interpolate between the two.

Many of the definitions of intrinsic dimension are consistent.
This means that it is possible to compute a scale $\epsilon$ from $\Omega$ for which the intrinsic dimension of each cluster converges to the dimension of $V$ if $m$ is sufficiently large and $\Omega$ is sampled sufficiently densely.
By contrast, our paradigm is that~$m$ is fixed. For us,~$m$ does not tend to infinity. Our standing assumption is that we are given one fixed sample~$\Omega$. The goal is to compute a meaningful
dimension from that fixed sample of $m$ points. For this reason, we cannot unreservedly employ results on appropriate parameters $\epsilon$ in our methods. The sample $\Omega$ will
almost never satisfy the assumptions  that are needed.
Our approach to circumvent this problem is to create a \emph{dimension diagram}. Such diagrams are
 shown in Figures \ref{DD}, \ref{DDSO3}, \ref{DD2times3} and \ref{dim_cyclo_plot}.

\begin{definition}
Let $\mathrm{dim}(\Omega,\epsilon)$ be one of the subsequent dimension estimates. The \emph{dimension diagram} of the sample $\Omega$ is the graph of
the function $\,(0,1]\to \mathbb{R}_{\geq 0},\, \epsilon\mapsto \mathrm{dim}(\Omega,\epsilon)$.
\end{definition}
\begin{remark} The idea of using dimension diagrams is inspired by persistent homology. Our dimension diagrams and our
persistent homology barcodes of Section \ref{sec:homology} both use
$\epsilon $ in the interval $ [0,1]$ for the horizontal axis.
This uniform scale for all samples $\Omega$ makes comparisons across different datasets easier.
\end{remark}

The true dimension of a variety is an integer. However,
we defined the dimension diagram to be the graph of a function whose range
 is a subset of the real numbers.
 The reason is that the subsequent estimates  do not return integers.  A noninteger dimension can be meaningful mathematically, such as in the case of a fractal curve which fills space densely enough that its dimension could be considered closer to $2$ than $1$.  By plotting these
diagrams, we hope to gain information about the true dimension $d$
of the variety $V$ from which $\Omega$ was sampled.

\begin{figure}[h]
\begin{center}
\includegraphics[scale=0.34]{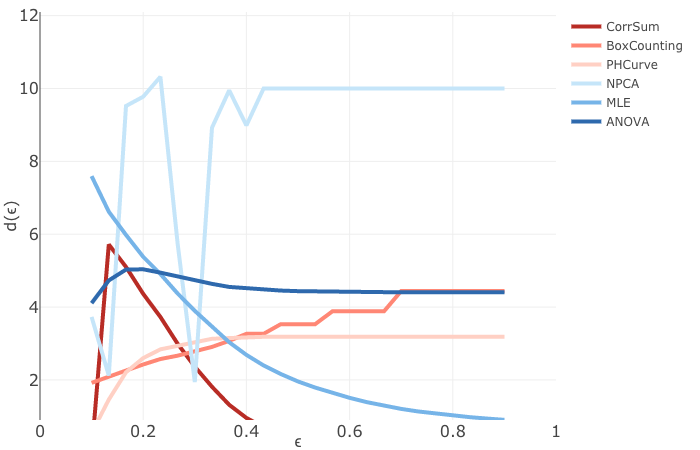}\includegraphics[scale=0.34]{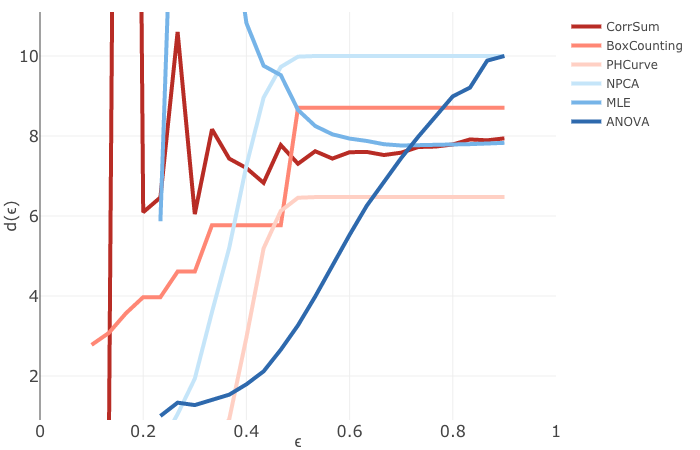}
\caption{Dimension diagrams for 600 points on the variety of $3 \times 4$
 matrices of rank $2$. This is a  projective variety of dimension $9$. Its affine cone has dimension $10$.
 The left picture shows dimension diagrams for the estimates in Euclidean space  $\RR^{12}$.
 The right picture shows those for projective space $\PP^{11}_\RR$.
 The projective diagrams yield better estimates.
The 600 data points were obtained by independently sampling pairs of $4\times 2$ and $2\times 3$ matrices, each with independent entries from the normal distribution, and then multiplying them. }\label{DD}
\end{center}
\end{figure}

One might be tempted to use the same dimension estimate for
$\RR^n$ and $\PP^{n-1}_\RR$, possibly via the Euclidean distance on an affine patch of
$\mathbb{P}_\RR^{n-1}$. However, the Theorema Egregium by Gauss implies that any projection from
$\mathbb{P}_\RR^{n-1}$ to $\RR^{n-1}$ must distort lengths. Hence, because we gave the parameter $\epsilon$ a metric meaning, we must be careful and treat real Euclidean space and real projective space separately.

Each of the curves seen in Figure~\ref{DD} is a dimension diagram. We used six different methods for estimating the dimension on a fixed sample of $600$ points.
For the horizontal axis on the left  we took the distance (\ref{Euclidean_scaled}) in $\RR^{12}$.
For the diagram on the right we took  (\ref{FS_scaled}) in $\PP^{11}_\RR$.

\subsection{Six dimension estimates}\label{sub:dimension}

In this section, we introduce  six dimension estimates.
  They are adapted from the existing literature.
 Figures \ref{DD}, \ref{DDSO3}, \ref{DD2times3} and~\ref{dim_cyclo_plot} show  dimension diagrams
generated by our implementation. Judging from those figures,
the estimators CorrSum, PHCurve, MLE and ANOVA all perform well on each of the examples. By contrast, NPCA and BoxCounting frequently overestimate the dimension. In general, we found it useful to allow for a ``majority vote" for the dimension. That is, we choose as dimension estimate the number which is closest to most of the estimators for a significant (i.e.~``persistent'') range of $\epsilon$-values in $[0,1]$.

\medskip
\noindent
{\bf NPCA Dimension}.
The gold standard of dimension estimation is
PCA.
Assuming that $V$ is a linear subspace of $\RR^n$,
we perform the following steps for the input $\Omega$.
First, we record the \emph{mean}  $\,\overline{u} := \frac{1}{m} \sum_{i=1}^m u^{(i)}$. Let $M$ be the $m \times n$-matrix with rows $u^{(i)}  - \overline{u}$.
We compute $\sigma_1\geq \cdots\geq \sigma_{\min\{m,n\}}$, the \emph{singular values} of $M$. The  {\em PCA dimension} is the number of $\sigma_i$ above a certain threshold. For instance, this threshold could be the same as in the definition of the numerical rank in (\ref{eq:numericalrank}) below.
 Following \cite[p. 30]{LV}, another idea is to set the threshold as $\sigma_k$, where ${k = {\rm argmax}_{1\leq i \leq \min\{m,n\}-1} \vert\log_{10}(\sigma_{i+1}) - \log_{10}(\sigma_{i})\vert}$. In our experiments we found that this improved the dimension estimates. In
some situations it is helpful to further divide each column of $M$ by its standard deviation.
This approach is explained in \cite[p. 26]{LV}.

Using PCA on a local scale is known as
{\em Nonlinear Principal Component Analysis} (NPCA).
Here we partition the sample $\Omega$ into $l$ clusters~$\Omega_1^\epsilon,\ldots,\Omega_l^\epsilon\subset\Omega$ depending on $\epsilon$. For each~$\Omega_i^\epsilon$ we apply the usual PCA and obtain
the estimate~${\rm dim}_{\rm pca}(\Omega_i^\epsilon)$. The idea behind this is that
the manifold $V \backslash {\rm Sing}(V)$ is approximately linear locally.
We take the average of these local dimensions, weighted by
the size of each cluster. The result is the {\em nonlinear PCA dimension}
\begin{equation}
\label{eq:npca}
 {\rm dim}_{\rm npca}(\Omega,\epsilon)\, \,:= \,\,\, \frac{1}{\sum_{i=1}^l |\Omega_i^\epsilon|}\,\,
 \sum_{i=1}^l |\Omega_i^\epsilon| \cdot  {\rm dim}_{\rm pca}(\Omega_i^\epsilon).
\end{equation}

Data scientists have many clustering methods.
For our study we use \emph{single linkage clustering}. This works as follows.
The clusters are the connected components in  the graph with vertex set $\Omega$ whose
edges are the pairs of points having distance at most $\epsilon$.
We do this either in Euclidean space with metric
(\ref{Euclidean_scaled}), or in projective space with metric (\ref{FS_scaled}).
In the latter case, the points come from the cone over the true variety $V$. To make $\Omega$ less scattered, we sample a random linear function $l$ and scale each data point $u^{(i)}$ such that $l(u^{(i)})=1$. Then we use those affine coordinates for NPCA. We chose this procedure because NPCA detects linear spaces and the proposed scaling maps projective linear spaces to affine-linear spaces.

We next introduce the notions of box counting dimension,
 persistent homology curve dimension
and correlation dimension. All three of these belong to the class of
\emph{fractal-based methods}, since they rest  on the idea of using the fractal
dimension as a proxy for ${\rm dim}(V)$.

\medskip
\noindent
{\bf Box Counting Dimension}.
Here is the geometric idea in $\mathbb{R}^2$.
Consider a square of side length~$1$ which we cover by miniature squares.  We could cover it with $4$ squares of side length $\frac{1}{2}$, or $9$ squares of side length $\frac{1}{3}$, etc. What remains constant is
the log ratio of the number of pieces over the magnification factor.
 For the square: $\frac{\log(4)}{\log(2)}=\frac{\log(9)}{\log(3)}=2$. If~$\Omega$
    only intersects $3$ out of $4$ smaller squares,
    then we estimate the dimension to be between $1$ and $2$.

In $\mathbb{R}^n$ we choose as a box the parallelopiped with lower vertex
${u^{-} = {\rm min}(u^{(1)},\ldots,u^{(m)})}$ and upper vertex
${u^{+}= {\rm max}(u^{(1)},\ldots,u^{(m)})}$, where ``min'' and ``max''
are  coordinatewise minimum and maximum. Thus the box is
$\{ x \in \RR^n\,: \, u^- \leq x \leq u^+ \}$.
For $j=1,\ldots,n$, the
interval $[u^-_j,u^+_j]$ is divided into $R(\epsilon)$ equally sized intervals, whose length depends on $\epsilon$.
A $d$-dimensional object is expected to capture $R(\epsilon)^d$ boxes.
We determine the number $\nu$
 of boxes that contain a point in $\Omega$. Then the {\em box counting dimension estimate} is
\begin{equation}
\label{eq:boxdim}
\dim_{\rm box}(\Omega,\epsilon) \,\, := \,\, \frac{{\rm log}(\nu)}{ {\rm log}(R(\epsilon))}.
\end{equation}
How to define the function $R(\epsilon)$? Since the number of small boxes is very large,
we cannot iterate through all boxes. It is desirable to decide from a data point $u\in \Omega$
 in which box it lies. To this end, we set
  $R(\epsilon) = \lfloor \frac{\lambda}{\epsilon} \rfloor + 1$, where $\lambda:= \max_{1\leq j\leq n} \vert u^+_j - u^-_j \vert$. Then, for $u\in\Omega$ and $k=1,\ldots,n$ we compute the largest $q_k$ such that  $\frac{q_k}{R(\epsilon)} \vert u^+_k - u^-_k\vert\leq \vert u_k-u^-_k\vert$. The $n$  numbers $q_1,\ldots,q_n$ completely determine the box that contains the sample $u$.

For the box counting dimension in real projective space, we represent the points in $\Omega$ on an affine patch
of $\PP^{n-1}_\RR$. On this patch we do the same construction as above, the only exception being that ``equally sized intervals'' is measured in terms of scaled Fubini-Study~distance (\ref{FS_scaled}).

\medskip
\noindent
{\bf  Persistent Homology Curve Dimension}.
The underlying idea was proposed by the Pattern Analysis Lab at Colorado State University~\cite{PALCSU}.
First we partition $\Omega$ into $l$ clusters~$\Omega_1^\epsilon,\ldots,\Omega_l^\epsilon$ using single linkage clustering with $\epsilon$. On each subsample $\Omega_i$ we construct a minimal spanning tree. Suppose that the cluster $\Omega_i$ has $m_i$ points. Let  $l_i(j)$ be the length of the $j$-th longest edge in a minimal spanning tree for $\Omega_i$.
For each $\Omega_i$ we compute
\[
\dim_{\rm PHcurve}(\Omega_i,\epsilon) = \left\vert \frac{\log(m_i)}{\log(\frac{1}{m_i-1}\sum_{j=1}^{m_i-1}l_i(j))  } \right\vert.
\]
The {\em persistent homology curve dimension estimate} $\,\dim_{\rm PHCurve}(\Omega,\epsilon)\,$
is the average of the local dimensions, weighted by
the size of each cluster:
 \begin{equation*}
\dim_{\rm PHcurve}(\Omega,\epsilon) \,\, := \,\,
\frac{1}{\sum_{i=1}^l \vert \Omega_i^\epsilon\vert}\sum_{i=1}^m \vert\Omega_i\vert\dim_{\rm PHcurve}(\Omega_i,\epsilon) . \end{equation*}
In the clustering step we take the distance~(\ref{Euclidean_scaled})
if the variety is affine and  (\ref{FS_scaled}) if it is projective.

\medskip
\noindent
{\bf  Correlation Dimension}.
This is motivated as follows.
Suppose that $\Omega$ is uniformly distributed in the unit ball.
 For pairs $u,v\in\Omega$, we have $\mathrm{Prob}\{\mathrm{dist}_{\mathbb{R}^n}(u,v) < \epsilon\} = \epsilon^d$, where $d = \mathrm{dim}(V)$.
  We set  $C(\epsilon) := (1/\tbinom{m}{2})\cdot\sum_{1\leq i<j\leq m} \mathbf{1}(\mathrm{dist}_{\mathbb{R}^n}(u^{(i)},u^{(j)}) < \epsilon)$ ,
 where $\mathbf{1}$ is the indicator function.
Since we expect
the empirical distribution $C(\epsilon)$ to be approximately $\epsilon^d$,
 this suggests using
  $\,\frac{\log(C(\epsilon))}{ {\rm log}(\epsilon)}\,$ as  dimension estimate.
In \cite[\S 3.2.6]{LV} it is mentioned that a more practical estimate is obtained from $C(\epsilon)$ by selecting some small $h>0$ and putting
\begin{equation} \label{eq:cordim} \qquad  \quad \dim_{\rm cor}(\Omega,\epsilon)\,\, := \,\,
\left\vert \frac{\log C(\epsilon) - \log C(\epsilon + h)}{\log(\epsilon) - \log(\epsilon + h)}\right\vert.
\end{equation}
In practice, we compute the dimension estimates for a finite subset of parameters
$\epsilon_1,\ldots, \epsilon_k$ and  put $h = \min_{i\neq j} \vert \epsilon_i -\epsilon _j\vert$.
The ball in $\PP^{n-1}_\RR$ defined by the scaled Fubini-Study distance (\ref{FS_scaled}) is a
  spherical cap of radius~$\epsilon$. Its volume relative to a cap of radius 1 is $\int_0^\epsilon (\sin \alpha)^{d-1} \mathrm{d}\alpha/\int_0^1 (\sin \alpha)^{d-1} \mathrm{d}\alpha$, which we approximate by
  $\,\bigl(\frac{\sin(\epsilon)}{\sin(1)} \bigr)^{d}$. Hence, the {\em projective correlation dimension estimate} is
  \begin{equation*}
\dim_{\rm cor}(\Omega,\epsilon) \,\, := \,\, \left\vert \frac{\log  C(\epsilon) - \log C(\epsilon + h)}{\log(\sin(\epsilon)) - \log(\sin(\epsilon + h))}\right\vert, \qquad \qquad \qquad
\end{equation*}
with the same $h$ as above and where $ C(\epsilon)$ is now computed using the Fubini-Study distance.

\smallskip

We next describe two more methods. They differ from the aforementioned in
that they derive from estimating the dimension of the variety~$V$ locally at a \emph{distinguished point} $u^{(\star)}$.

\medskip
\noindent
{\bf MLE Dimension}.
Levina and Bickel \cite{LB} introduced a maximum likelihood estimator for the dimension of
an unknown variety
$V$. Their estimate is derived for samples in Euclidean space $\mathbb{R}^n$.
Let $k $ be the number of samples $u^{(j)}$
in $\Omega$ that are within distance~$\epsilon$ to~$u^{(\star)}$.
We write $T_i(u^{(\star)})$ for the distance from $u^{(\star)}$ to its $i$-th nearest neighbor  in $\Omega$. Note that $T_k(u^{(\star)}) \leq \epsilon < T_{k+1}(u^{(\star)})$.
The {\em Levina-Bickel formula} around the point $u^{(\star)}$ is
\begin{equation}\label{BL_estimator}
\dim_{\rm MLE}(\Omega,\epsilon,u^{(\star)}) \,\, := \,\,
\left( \frac{1}{k} \sum_{i=1}^k \log \frac{\epsilon}{T_i(u^{(\star)})} \right)^{-1}.
 \end{equation}
 This expression is derived from the hypothesis that
 $k = k(\epsilon)$ obeys a Poisson process on the  $\epsilon$-neighborhood ${\{u\in\Omega : \mathrm{dist}_{\mathbb{R}^n}(u,u^{(\star)}) \leq \epsilon\}}$, in which $u$ is uniformly distributed.
 The formula~(\ref{BL_estimator}) is obtained by solving the likelihood equations
 for this Poisson process.

 In projective space, we model $k(\epsilon)$ as a Poisson process on $\{u\in\Omega : \mathrm{dist}_{\mathbb{P}_\mathbb{R}^{n-1}}(u,u^{(\star)}) \leq \epsilon\}$. However, instead of assuming that $u$ is uniformly distributed in that neighborhood, we assume that the orthogonal projection of $u$ onto the tangent space~$\mathrm{T}_{u^{(\star)}} \mathbb{P}_\mathbb{R}^{n-1}$ is uniformly distributed in the associated ball of radius $\sin{\epsilon}$. Then, we derive the formula
$$
\dim_{\rm MLE}(\Omega,\epsilon,u^{(\star)}) \,\, := \,\,
\left( \frac{1}{k} \sum_{i=1}^k \log \frac{\sin(\epsilon)}{\sin(\widehat{T}_i(u^{(\star)}))} \right)^{-1},
$$
 where $\widehat{T}_i(u^{(\star)})$ is the distance from $u^{(\star)}$ to its $i$-th nearest neighbor
  in $\Omega$ measured for (\ref{FS_scaled}).

It is  not clear how to choose $u^{(\star)}$ from the given $\Omega$.
We chose the following method. Fix the sample neighborhood
$\,\Omega_i^\epsilon:= \{u\in\Omega : \mathrm{dist}_{\mathbb{R}^n}(u,u^{(i)}) \leq \epsilon\}$. For each $i$ we     evaluate the formula (\ref{BL_estimator})   for $\Omega_i^\epsilon$ with distinguished point $u^{(i)}$. With this, the {\em MLE dimension estimate} is
 \begin{equation*}
 \dim_{\rm MLE}(\Omega,\epsilon) \,\,:= \,\,
 \frac{1}{\sum_{i=1}^m \vert \Omega_i^\epsilon\vert}\sum_{i=1}^m \vert\Omega_i^\epsilon \vert \cdot
 \dim_{\rm MLE} (\Omega_i^\epsilon,\epsilon, u^{(i)}).
 \end{equation*}

\medskip
\noindent
{\bf ANOVA Dimension}.
 Diaz, Quiroz and Velasco \cite{DQV} derived an
analysis of variance estimate for the dimension of $V$.
In their approach, the following expressions are important:
\begin{equation}\label{formula_for_the_betas}
    \beta_{2s-1}  \,=\,\frac{\pi^2}{4} - 2\sum_{j=0}^s\frac{1}{(2j+1)^2} \quad \text{and} \quad
    \beta_{2s}  \,=\,\frac{\pi^2}{12} - 2\sum_{j=0}^s\frac{1}{(2j)^2} \qquad \text{for}  \quad
    s\in \mathbb{N}.
\end{equation}
The quantity $\beta_d$ is the variance of the random variable $\Theta_d$,
 defined as the angle between
two uniformly chosen random points on the $(d-1)$-sphere.
We again fix $\epsilon > 0$, and we relabel so that
 $u^{(1)},\ldots,u^{(k)}$ are the points in $\Omega$ with distance at most~$\epsilon$ from $u^{(\star)}$.
 Let $\theta_{ij} \in [0,\pi]$ denote the angle between
  $u^{(i)} - u^{(\star)}$ and  $u^{(j)}- u^{(\star)}$. Then, the  {\em sample covariance} of the $\theta_{ij} $ is
 \begin{equation}\label{S_statistic}
S \,\,=\,\, \frac{1}{\tbinom{k}{2}} \sum_{1\leq i<j\leq k} \left(\theta_{ij} - \frac{\pi}{2}\right)^2.
\end{equation}
The analysis in \cite{DQV} shows that, for small $\epsilon$ and $\Omega$ sampled from a $d$-dimensional
manifold, the angles $\theta_{ij}$ are approximately $\Theta_d$-distributed. Hence, $S$ is expected to be
close to $\beta_{\dim V}$. The {\em ANOVA dimension estimate}
of  $\Omega$ is the index $d$ such that
$\beta_d$ is closest to $S$:
\begin{equation}
\label{eq:mauricio}
 \dim_{\rm ANOVA}(\Omega,\epsilon, u^{(\star)}) \,\, := \,\, {\rm argmin}_d \, |  \beta_d - S | .
\end{equation}
As for the MLE estimate, we average (\ref{eq:mauricio})
 over all $u\in\Omega$ being the distinguished point.

To transfer the definition to projective space, we revisit the idea behind the ANOVA estimate. For $u$ close to $u^{(\star)}$, the secant through
$u$ and~$u^{(\star)}$ is approximately parallel to the tangent space of $V$ at $u^{(\star)}$. Hence, the unit vector $(u^{(\star)}-u)/ \Vert u^{(\star)}-u\Vert$ is close to being in the tangent space $\mathrm{T}_{u^{(\star)}}(V)$. The sphere in $\mathrm{T}_{u^{(\star)}}(V)$ has dimension $\mathrm{dim} \,V - 1$ and we know the variances of the random angles $\Theta_{d}$. To mimic this construction in $\mathbb{P}_\RR^{n-1}$ we use the angles between geodesics meeting at $u^{(\star)}$. In our implementation,
we orthogonally project $\Omega$ to the tangent space~$\mathrm{T}_{u^{(\star)}} \mathbb{P}_\RR^{n-1}$ and compute (\ref{S_statistic}) using coordinates on that space.

\smallskip

We have defined all the mathematical ingredients inherent
in our dimension diagrams. Figure~\ref{DD} now makes sense.
Our software and its
applications will be discussed in Section \ref{sec:experiments_and_data}.

\section{Persistent Homology}\label{sec:homology}

This section connects algebraic geometry and topological data analysis. It concerns the
computation and analysis of the {\em persistent homology} \cite{Carlsson} of our sample $\Omega$. Persistent homology
of~$\Omega$ contains information about the shape of the unknown variety $V$
from which~$\Omega$~originates.

\subsection{Barcodes}
Let us briefly review the idea. Given $\Omega$, we associate a simplicial complex
with each  value of a parameter $\epsilon \in [0,1]$. Just like in the case of
the dimension diagrams in the previous section, $\epsilon$ determines the scale at which we consider $\Omega$ from local ($\epsilon = 0$) to global ($\epsilon = 1$).
The complex at~$\epsilon =0$ consists of only the vertices and at $\epsilon=1$ it is the full simplex on $\Omega$.

Persistent homology identifies and keeps track of the changes in the homology of those complexes as
$\epsilon$ varies. The output is a \emph{barcode,} i.e.\ a collection of intervals.
Each interval in the barcode corresponds to a topological feature which appears at the value of a parameter given by the left hand endpoint of the interval and disappears at the value given by the right hand endpoint. These barcodes play the same role as a histogram does in summarizing the shape of the data, with long intervals corresponding to strong topological signals and short ones to noise. By plotting the intervals we obtain a barcode, such as the one in Figure \ref{fig:trott2}.

\begin{figure}[h]
  \begin{center}
   \includegraphics[width=0.35\textwidth]{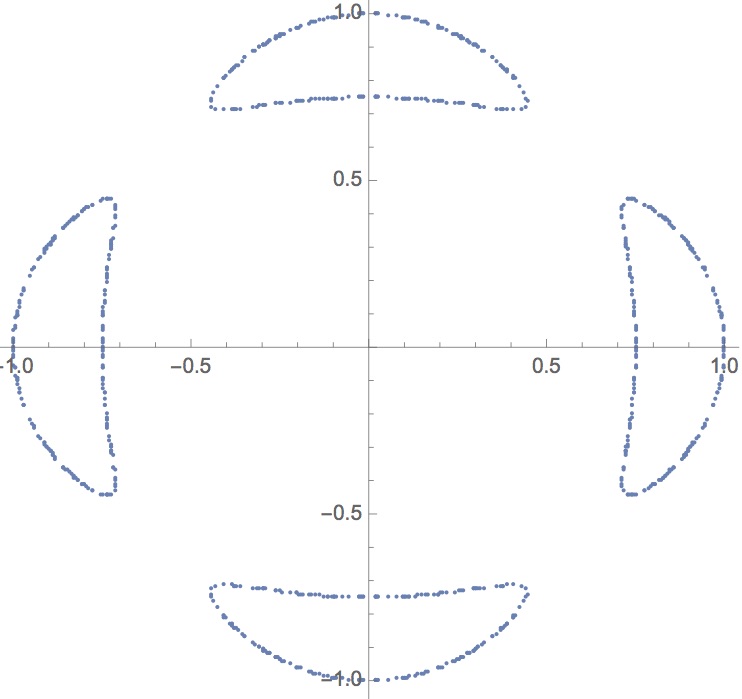} \hspace{1cm}  \includegraphics[width=0.45\textwidth]{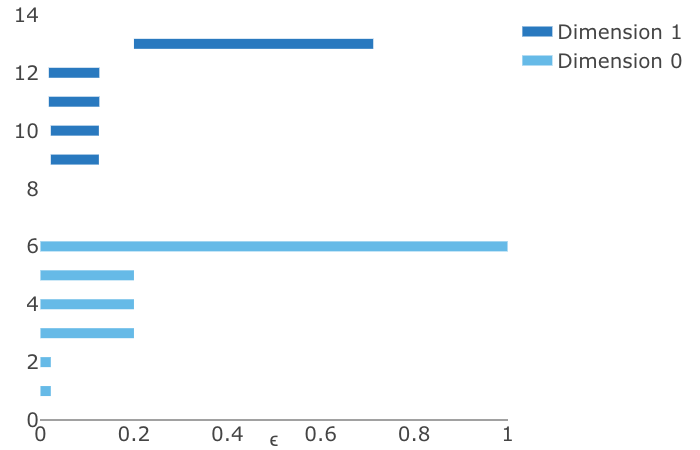}
 \end{center}
 \caption{Persistent homology barcodes for the Trott curve. }\label{fig:trott2}
   \end{figure}

 The most straightforward way
to associate a simplicial complex to $\Omega$ at $\epsilon$
is by covering~$\Omega$ with open sets $\,U(\epsilon)=\bigcup_{i=1}^m
U_i(\epsilon)$ and then building the associated {\em nerve complex}. This is the simplicial complex with vertex set
$[m] = \{1,2,\ldots,m\}$, where a subset  $\sigma$ of $ [m]$
is a face if and only if $\bigcap_{i \in \sigma} U_i(\epsilon) \not= \emptyset $.
If  all nonempty finite intersections of $U_i(\epsilon)$ are contractible topological spaces,
then the Nerve Lemma guarantees that the homology groups of $U(\epsilon)$
agree with those of its nerve complex.
When $U_i(\epsilon)$ are $\epsilon$-balls around the data points, {\it i.e.}
 \begin{equation}\label{def_balls} U_i(\epsilon) \,:= \,\{v\in \mathbb{R}^n: \mathrm{dist}_{\mathbb{R}^n}(u^{(i)},v) < \epsilon\}
 \, \text{ or } \,
 U_i(\epsilon)  \,:=\, \{v\in \mathbb{P}_\RR^{n-1}: \mathrm{dist}_{\mathbb{P}_\RR^{n-1}}(u^{(i)},v) < \epsilon\},\end{equation}
 the nerve complex is called the {\em \v{C}ech complex} at $\epsilon$.  Here
 $\mathrm{dist}_{\mathbb{R}^n}$ and $\mathrm{dist}_{\mathbb{P}_\RR^n}$ are the distances from~(\ref{Euclidean_scaled}) and (\ref{FS_scaled}), respectively.
 Theorem \ref{niyogi} gives a precise statement for a sufficient condition under which the \v{C}ech complex of $U(\epsilon)$ built on $\Omega$ yields the correct topology of $V$. However, in practice the hypotheses of the theorem will rarely be  satisfied.

\v{C}ech complexes are computationally demanding as they require storing simplices in different dimensions. For this reason, applied topologists prefer to work with the {\em Vietoris-Rips complex},
which is the flag simplicial complex determined by the edges of the \v{C}ech complex.
This means that  a subset $\sigma\subset[m]$ is a face  of the Vietoris-Rips complex
 if and only if~$\,U_i(\epsilon) \bigcap U_j(\epsilon) \not= \emptyset\,$  for all $i,j \in \sigma$.
 With the definition in  (\ref{def_balls}), the balls
 $\,U_i(\epsilon) $ and~$U_j(\epsilon) $ intersect if and only if their centers $u^{(i)}$ and~$u^{(j)}$ are less than $2 \epsilon$ apart.

Consider the sample from the Trott curve in Figure~\ref{fig:trott2}.
Following Example~\ref{ex:trottsample}, we sampled by selecting random~$x$-coordinates
  between $-1$ and $1$,  and  solving for $y$, or vice versa.
The picture on the right shows the barcode. This was computed via the Vietoris-Rips complex. For dimensions $0$ and~$1$ the six longest bars are displayed. The sixth bar in dimension $1$ is so tiny that we cannot see it. In the range where~$\epsilon$ lies between $0$ and $0.2$, we see four components.
The barcode for dimension~$1$ identifies four persisting features for $\epsilon$ between~$0.01$ and~$0.12$. Each of these indicates an oval. Once these disappear, another loop appears. This corresponds to the fact that the four ovals are arranged
to form a circle. So persistent homology picks up on both intrinsic and extrinsic
topological features of the Trott curve.

The repertoire of algebraic geometry offers a fertile testing ground
for practitioners of persistent homology. For many classes of algebraic
varieties, both over $\RR$ and $\CC$, one has a priori information
about their topology. For instance, the determinantal variety in Example~\ref{ex:toblerone}
is the $3$-manifold $\mathbb{P}_{\mathbb{R}}^1 \times \mathbb{P}_{\mathbb{R}}^2$.
Using Henselman's software {\tt Eirene} for persistent homology \cite{Eirene}, we computed
 barcodes for several samples $\Omega$ drawn
from varieties with known topology.

\subsection{Tangent Spaces and Ellipsoids}

We underscore the benefits of an algebro-geometric perspective by proposing
a variant of persistent homology that performed well in the examples we tested.
Suppose that, in addition to knowing $\Omega$ as a finite metric space,
we also have information on the  tangent spaces of the unknown variety $V$ at the points
$u^{(i)}$.   This will be the case after we have learned some
polynomial equations for $V$ using the methods in
Section \ref{sec:finding_equations}.
In such circumstances, we suggest replacing the
$\epsilon$-balls in (\ref{def_balls}) with
 \emph{ellipsoids} that are aligned to the tangent spaces.

 The motivation is that
in a variety with a bottleneck, for example in the shape of a dog bone, the balls around points on the bottleneck may intersect for $\epsilon$ smaller than that which is necessary for the full cycle to appear. When $V$ is a manifold, we
 design a covering of $\Omega$ that exploits the locally linear structure. Let $0<\lambda<1$. We take $U_i(\epsilon)$ to be an ellipsoid around $u^{(i)}$
   with principal axes of length $\epsilon$ in the tangent direction of $V$ at $u^{(i)}$ and principal axes of length $\lambda\epsilon$ in the normal direction. In this way, we allow ellipsoids to intersect with their neighbors and thus reveal the true homology of the variety before ellipsoids intersect with other ellipsoids across the medial axis.  The parameter $\lambda$ can be chosen by the user. We believe that $\lambda$ should be proportional to the \emph{reach}
   of $V$. This metric invariant is defined in the next subsection.

In practice, we perform the following procedure. Let $f = (f_1,\ldots,f_k)$ be a vector of polynomials that vanish on
$V$, derived from
the sample $\Omega\subset\mathbb{R}^n$ as in
Section \ref{sec:finding_equations}. An estimator for the tangent space $\mathrm{T}_{u^{(i)}}V$
is the kernel of the Jacobian matrix of $f$ at $u^{(i)}$. In symbols,
\begin{equation}
\label{eq:jacobiankernel}
 \widehat{\mathrm{T}}_{u^{(i)}}V \,\,:= \,\, \ker Jf(u^{(i)}).
\end{equation}
 Let  $q_i$ denote the quadratic form on $\RR^n$ that takes value $1$ on
 $\,\widehat{\mathrm{T}}_{u^{(i)}}V \cap \mathbb{S}^{n-1}\,$ and value $\lambda$ on the orthogonal complement of $\widehat{\mathrm{T}}_{u^{(i)}}V$ in the sphere $\mathbb{S}^{n-1}$.
 Then, the $q_i$ specify the ellipsoids
 $$E_i \,\, := \,\, \bigl\{\sqrt{q_i(x)} \,x \in\mathbb{R}^n\,:\, \Vert x\Vert \leq 1\bigr\}. $$
 The role of the $\epsilon$-ball  enclosing the $i$th sample point
is now played by
  $U_i(\epsilon) := u^{(i)} + \epsilon E_i$.
 These ellipsoids determine the
 covering  $\,U(\epsilon) = \bigcup_{i=1}^m U_i(\epsilon)\,$
  of the given point cloud $\Omega$. From this covering we construct
  the associated \v{C}ech  complex or Vietoris-Rips complex.

While using ellipsoids is appealing, it has practical drawbacks.
  Relating the smallest $\epsilon$ for which $U_i(\epsilon)$ and $U_j(\epsilon)$ intersect to
  $\mathrm{dist}_{\RR^n}(u^{(i)},u^{(j)})$ is not easy. For this reason we implemented the following variant of ellipsoid-driven barcodes.
   We use the simplicial complex on $[m]$ where
 \begin{equation}\label{ellipsoids_relaxation}
  \sigma \text{ is a face iff }\,
  \frac{\mathrm{dist}_{\mathbb{R}^n} (u^{(i)}, u^{(j)})}{\frac{1}{2}(\sqrt{q_i(h)}+\sqrt{q_j(h)}\,)} < 2\epsilon\,
\text{ for all } i,j \in \sigma,
  \text{ where } h = \frac{u^{(i)} - u^{(j)}}{\Vert u^{(i)} - u^{(j)}\Vert}.\end{equation}
 In (\ref{ellipsoids_relaxation}) we weigh the distance between $u^{(i)}$ and $u^{(j)}$ by the arithmetic mean of the radii of the two ellipsoids $E_i$ and $E_j$ in the direction $u^{(i)} - u^{(j)}$.
 If all quadratic forms $q_i$ were equal to $\sum_{j=1}^n x_j^2$, then
 the simplicial complex  of  (\ref{ellipsoids_relaxation}) equals the
  Vietoris-Rips complex  from (\ref{def_balls}).

 \begin{figure}[h!]
   \begin{center}
    \includegraphics[width=0.45\textwidth]{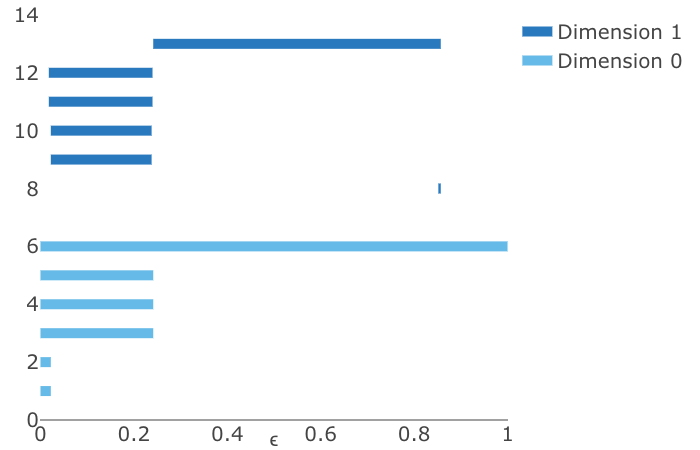}
    \hspace{1cm}
     \includegraphics[width=0.45\textwidth]{trott_curve.png}
  \end{center}
   \caption{The left picture shows the barcode constructed from the ellipsoid-driven
simplicial  complex (\ref{ellipsoids_relaxation}) with $\lambda = 0.01$,
    for the sample from the Trott curve used in~Figure~\ref{fig:trott2}.
     For comparison we display the barcode from Figure \ref{fig:trott2} in the right picture. All relevant topological features persist longer in the left plot.\label{ellipsoid_figure}}
 \end{figure}

Figure \ref{ellipsoid_figure} compares the barcodes for the classical Vietoris-Rips complex with those obtained from ellipsoids. It seems promising to further
develop variants of persistent homology that take some of the defining polynomial equations for
$(\Omega,V)$ into consideration.

\subsection{Reaching the Reach} \label{reach}

The \v{C}ech complex of a covering $U=\bigcup_{i=1}^m U_i$ has the homology of the union of balls $U$. But, can we give conditions on the sample $\Omega\subset V$ under which a  covering reveals the true homology of $V$?
A  result due to
Niyogi, Smale and Weinberger (Theorem~\ref{niyogi} below) offers an answer in
some circumstances. These involve the
concept of the \emph{reach}, which is an important metric invariant of a variety~$V$.
We here focus on varieties $V$ in the
 Euclidean space $\RR^n$.

\begin{definition}
 The {\em medial axis} of $V$ is the set
 $M_V$ of all points $u \in \RR^n$  such that the minimum
  distance from $V$ to $u$ is attained by
 two distinct points.
  The {\em reach}  $\tau(V)$
  is the infimum of all distances from points on the variety $V$ to any point in its medial axis $M_V$. In formulas: $\tau(V):=\inf_{u\in V, w \in M_V}\Vert u-w\Vert$. If $M_V = \emptyset$, we define $\tau(V) = +\infty$.
\end{definition}
Note that $\tau(V)=+\infty$, if and only if $V$ is an affine-linear subspace.
Otherwise, the reach is a non-negative real number.
In particular, there exist varieties $V$ with $\tau(V)=0$. For instance, consider the union of two lines $\,V=\{(x,y)\in\mathbb{R}^2 : xy=0\}$. All points in the diagonal $D=\{(x,y)\in\mathbb{R}^2 : x=y, x\neq 0\}$ have two
closest points on $V$. Hence, $D$ is a subset of the medial axis $M_V$, and we conclude that
  $0\leq \tau(V)\leq \inf_{u\in V, w \in D} \Vert u-w\Vert =0$. In general, any singular
  variety with an ``edge'' has zero reach.

To illustrate
the concept of the reach, let $V$ be a smooth curve in the plane, and draw
the normal line at each point of $V$. The collection of these lines
is the {\em normal bundle}. At a short distance from the curve,
the normal bundle is a product: each point $u$ near $V$
has a unique closest point $u^*$ on $V$, and $u$ lies on the
normal line through $u^*$. At a certain distance, however,
some of the normal lines cross. If $u$ is a crossing point of minimal distance to~$V$,
then $u$ has no unique closest point $u^*$ on $V$. Instead, there are at least
two points on $V$ that are closest to $u$ and the distance from $u$ to each of them is the reach~$\tau(V)$. Aamari \emph{et al.}\ \cite{AKCMRW} picture this by writing that \emph{``one can roll freely a ball of radius $\tau(V)$ around $V$"}.

Niyogi,  Smale and Weinberger refer to $\tau(V)^{-1}$
as the ``condition number of $V$''. B\"urgisser \emph{et al.}\ \cite{BCL} relate $\tau(V)^{-1}$ to the condition number of a semialgebraic set. For the purposes of our survey it suffices to understand how the reach effects the quality of the covering $U(\epsilon)$. The following result is a simplified version of \cite[Theorem 3.1]{NSW},
suitable for low dimensions.
Note that Theorem \ref{niyogi} only covers those
varieties $V \subset \RR^n$ that are smooth and compact.

\begin{theorem}[Niyogi, Smale, Weinberger 2006]
\label{niyogi}
Let $V \subset \RR^n$ be a compact manifold of dimension $d \leq 17$, with reach $\tau = \tau(V)$
and $d$-dimensional Euclidean volume $\nu = {\rm vol}(V)$.
Let $\Omega= \{u^{(1)}, \ldots, u^{(m)} \}$ be i.i.d.~samples
drawn from the uniform probability measure on $V$.
Fix $\epsilon = \frac{\tau}{4}$ and $\,\beta = 16^d \tau^{-d} \nu\,$. For any desired $\delta > 0$, fix the sample size at
\begin{equation}
\label{eq:samplesize} m \,\,>\,\,\, \beta \cdot \bigl(\log(\beta)+d + \log(\frac{1}{\delta}) \bigr)   .
\end{equation}
With probability $\geq 1-\delta$, the homology
groups of the following set coincide with those of $V$:
$$U(\epsilon)\,\,=\,\,\bigcup_{i=1}^m  \,\bigl\{x\in \mathbb{R}^n : \Vert x-u^{(i)}\Vert <\epsilon \bigr\}.$$
\end{theorem}

A few remarks are in order. First of all, the theorem is stated using the Euclidean distance and not the scaled Euclidean distance (\ref{Euclidean_scaled}). However, scaling the distance by a factor $t$ means scaling the volume by $t^d$, so the definition of $\beta$ in the theorem is invariant under scaling. Moreover, the theorem has been rephrased in a manner that makes
it easier to evaluate the right hand side of (\ref{eq:samplesize}) in cases of interest.
The assumption $d\leq 17$ is not important: it ensures that the volume of the unit ball in $\RR^d$
can be bounded below by $1$.
Furthermore, in \cite[Theorem 3.1]{NSW}, the tolerance $\epsilon$ can be any
real number between $0$ and $\tau/2$, but then $\beta$ depends in a complicated manner
on $\epsilon$. For simplicity, we took $\epsilon = \tau/4$.

Theorem \ref{niyogi} gives the asymptotics of a
sample size $m$ that suffices to reveal all topological features of $V$. For concrete parameter values it is less useful, though. For example, suppose that
$V$ has dimension $4$, reach $\tau = 1$, and volume $\nu = 1000$.
If we desire a 90\% guarantee that $U(\epsilon)$ has the same homology as $V$, so
$\delta = 1/10$, then $m$ must exceed $1,592,570,365$.  In addition to that, the theorem assumes that the sample was drawn from the uniform distribution on $V$. But in practice one will rarely meet data that obeys such a distribution. In fact, drawing from the uniform distribution on a curved object is a non-trivial affair \cite{Diaconis}.

In spite of its theoretical nature, the Niyogi-Smale-Weinberger
formula is useful in that it highlights the importance of the reach $\tau(V)$
for analyzing point samples. Indeed, the dominant
quantity in (\ref{eq:samplesize}) is $\beta$, and this grows to
the power of $d$ in $\tau(V)^{-1}$. It is therefore of interest to
better understand $\tau(V)$ and to develop tools for estimating~it.

We found the following formula by Federer \cite[Theorem 4.18]{Fed} to be useful.
It expresses the reach of
a manifold $V$ in terms of points and their tangent spaces:
\begin{equation}
\label{eq:federer}
\qquad \tau(V)\,\,= \inf_{v \neq u \in V} \frac{||u-v||^2}{2\delta}, \quad \text{ where }\,\,
 \delta \, = \min_{x\in \mathrm{T}_vV} \Vert (u-v) - x\Vert.
 \end{equation}
This formula relies upon knowing the tangent spaces at each point of $u\in V$.

Suppose  we are given the finite sample $\Omega$ from $V$.
If some equations for $V$ are also known, then we can use the estimator
$\widehat{\mathrm{T}}_{u^{(i)}}V$ for the tangent space that was derived in (\ref{eq:jacobiankernel}).
From this we get the following formula for the {\em empirical reach} of our sample:
\[ \qquad \hat{\tau}(V)\,\, =\, \min_{u,v\in \Omega \atop u \not= v}
 \frac{||u-v||^2}{2\widehat{\delta}}, \quad \text{ where }\,\,
 \widehat{\delta} \,= \min_{x\in \widehat{\mathrm{T}}_vV} \Vert (u-v) - x\Vert.\]
A similar approach for estimating the reach was proposed by Aamari {\it et al.}  \cite[eqn.~(6.1)]{AKCMRW}.

\subsection{Algebraicity of Persistent Homology}

It is impossible to compute in the field of real numbers $\RR$.
Numerical computations employ floating point approximations.
These are actually rational numbers. Computing in algebraic
geometry has traditionally been centered around exact
symbolic methods. In that context,  computing with
algebraic numbers makes sense as well.
In this subsection we argue that, in the setting of this paper,
most  numerical  quantities in persistent homology, like the
barcodes and the reach, have an algebraic nature.
Here we assume that the variety $V$ is defined over~$\mathbb{Q}$.

We discuss the work of Horobe\c{t} and Weinstein in \cite{HW}
which concerns metric properties of a given variety $V \subset \RR^n$
that are relevant for its {\em true persistent homology}. Here, the true
persistent homology of $V$, at parameter value $\epsilon$,
refers to the homology of the $\epsilon$-neighborhood of $V$.
Intuitively, the true persistent homology of the Trott curve
is the limit of barcodes as in Figure  \ref{fig:trott2},
where more and more points are taken,
eventually filling up the entire curve.

An important player is the {\em offset hypersurface} $\mathcal{O}_\epsilon(V)$.
This is the algebraic boundary of the $\epsilon$-neighborhood of $V$.
More precisely, for any positive value of $\epsilon$,
the offset hypersurface  is the Zariski closure
of the set of all points in $\RR^n$ whose distance to $V$
equals~$\epsilon$.
If $n=2$ and $V$ is a plane curve, then the {\em offset curve} $\mathcal{O}_\epsilon(V)$
is drawn by tracing circles along~$V$.

\begin{figure}[h]
\begin{center}
\includegraphics[scale=1]{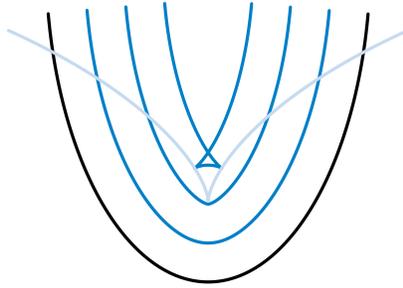}
\vspace{-0.2in}
\end{center}
\caption{
Offset curves (blue) and the evolute (light blue) of a conic (black).
\label{fig:offset}}
\end{figure}

\begin{example} \label{ex:ellipseoffset} \rm
In Figure~\ref{fig:offset} we examine a conic $V$, shown in black.
The light blue curve is its {\em evolute}. This is an {\em astroid} of degree $6$.
The evolute serves as the {\em ED discriminant} of $V$,
in the context seen in \cite[Figure 3]{DHOST}. The blue curves in Figure~\ref{fig:offset}
are the offset curves $\mathcal{O}_\epsilon(V)$. These have degree $8$ and
are smooth (over $\RR$) for small values of $\epsilon$. However, for larger
values of $\epsilon$, the offset curves are singular.
The transition point occurs at the cusp of the evolute.
\end{example}

It is shown in \cite[Theorem 3.4]{HW} that the endpoints of bars in
the true persistent homology of a variety~$V$
 occur at numbers  that are algebraic over
$\mathbb{Q}$.
The proof relies on results
in real algebraic geometry that
 characterize the family of fibers in a map
  of semialgebraic sets.

\begin{example} \rm
The bars of the barcode in Figure \ref{fig:trott2} begin and end near the  numbers
$$ \frac{1}{8} \,=\, 0.125 \,, \,\quad
\frac{\sqrt{  24025 - 217 \sqrt{9889}}}{248} \,=\, 0.19941426... \, , \,\quad
\frac{3}{4} \,=\, 0.75.
$$
These algebraic numbers delineate the true persistent homology of the Trott curve $V$.
\end{example}

The reach $\tau(V)$ of any real variety $V \subset \RR^n$
is also an algebraic number.
This follows from Federer's formula (\ref{eq:federer}) which expresses
$\tau(V)$ as the optimal value of a polynomial optimization problem.
In principle, the reach can be computed in exact arithmetic  from the polynomials that define $V$.
It remains an open problem how to do this effectively in practice.
Eklund's recent work on bottlenecks  \cite{Ekl} represents an important step
towards a solution.

At present we do not know a good formula or a tight bound for the
algebraic degrees of the barcode and the reach in terms
of the invariants of the variety $V$. Deriving such formulas
will require a further development and careful analysis
of the {\em offset discriminant} that was introduced in \cite{HW}.
We hope to return to this topic in the near future, as
it can play a fundamental link between
topology and algebraic geometry in the context of data science.

\section{Finding Equations}\label{sec:finding_equations}

Every polynomial in the ideal $I_V$ of the unknown variety $V$
vanishes on the sample $\Omega$.
The converse is not true, but it is reasonable to surmise that
it holds among polynomials of low degree. The ideal
$I_\Omega$ of the finite set $\Omega \subset \RR^n$ can be
computed using linear algebra.
 All our polynomials and ideals in this section lie in the ring
 $R = \RR[x_1,x_2,\ldots,x_n]$.

\subsection{Vandermonde Matrices}

 Let $\mathcal{M}$ be a finite linearly independent subset of $R$. We write
  $R_\mathcal{M}$ for the $\RR$-vector space with basis $\mathcal{M}$ and generally assume  that $\mathcal{M}$  is ordered, so that polynomials in $R_\mathcal{M}$ can be identified with vectors in $\RR^{\vert\mathcal{M}\vert}$.
  Two primary examples for $\mathcal{M}$ are the set of monomials
   ${\bf x}^e= x_1^{e_1} x_2^{e_2} \cdots x_n^{e_n}$
  of degree $d$ and the set of monomials of degree at most $d$. We
  use the notation $R_d$ and $R_{\leq d}$ for the corresponding
  subspaces of $R$. Their  dimensions $|\mathcal{M}|$ are
     $$ {\rm dim}(R_d)\,=\,  \binom{n+d-1}{d} \quad {\rm and} \quad
   {\rm dim}(R_{\leq d}) \,=\,  \binom{n+d}{d}. $$

 We write $U_\mathcal{M}(\Omega)$ for the $m \times |\mathcal{M}|$ matrix
 whose $i$-th row consists of the evaluations of the polynomials in $\mathcal{M}$ at the point $u^{(i)}$.
 Instead of $U_\mathcal{M}(\Omega)$ we write $U_d(\Omega)$ when $\mathcal{M}$ contains all monomials of degree $d$  and $U_{\leq d}(\Omega)$ when $\mathcal{M}$ contains monomials of degree $\leq d$.

 For example, if $n=1$, $m=3$, and $\Omega = \{u,v,w\}$ then
 $U_{\leq 3}(\Omega)$ is the  Vandermonde matrix
  \begin{equation}\label{Vandermonde} U_{\leq 3}(\Omega) =
  \begin{pmatrix}
\, u^3 & u^2 & u  & 1   \, \\
\, v^3 &  v^2 & v  & 1  \, \\
\, w^3 &  w^2 & w  & 1  \,
 \end{pmatrix} . \end{equation}
For $ n \geq 2$, we call $U_\mathcal{M}(\Omega)$ a  \emph{multivariate Vandermonde matrix}.
It has the following property:

\begin{remark}
The kernel of the multivariate Vandermonde matrix $U_\mathcal{M}(\Omega)$ equals the
vector space $\,I_\Omega \cap R_\mathcal{M}\,$ of all polynomials that
are linear combinations of $\mathcal{M}$ and  that vanish on $\Omega$.
\end{remark}

 The strategy for learning the variety $V$ is as follows.
 We hope to learn the ideal $I_V$ by making
  an educated guess for the set $\mathcal{M}$.
  The two desirable properties for $\mathcal{M}$ are:
 \begin{itemize}
 \item[(a)] The ideal $I_V$ of the unknown variety $V$ is
 generated by its subspace $I_V \cap R_\mathcal{M}$. \vspace{-0.1cm}
 \item[(b)] The inclusion of  $I_V \cap R_\mathcal{M}$ in its superspace
  $\,I_\Omega \cap R_\mathcal{M} = {\rm ker} (U_\mathcal{M}(\Omega))$
  is an equality.
  \end{itemize}
There is a fundamental tension between these two desiderata:
if $\mathcal{M}$ is too small then (a) will fail, and if
$\mathcal{M}$ is too large then (b) will fail.
But, of course, suitable sets $\mathcal{M}$ do always exist, since the
Hilbert's Basis Theorem ensures that all ideals in $R$ are finitely generated.

The requirement (b) imposes a lower bound on the size $m$ of the sample.
Indeed, $m$ is an upper bound on the
rank of $U_\mathcal{M}(\Omega)$, since that matrix has $m$ rows.
The rank of any matrix is equal to the number
of columns minus the dimension of the kernel. This implies:

\begin{lemma}
If (b) holds, then $\,m \geq |\mathcal{M}| - {\rm dim}(I_V \cap R_\mathcal{M})$.
\end{lemma}

In practice, however, the sample $\Omega$ is given and fixed.
Thus, we know $m$ and it cannot be increased.  The question is
 how to choose the set $\mathcal{M}$. This leads to
some interesting geometric combinatorics. For instance,
if we believe that $V$ is homogeneous with respect to some $\ZZ^r$-grading,
then it makes sense to choose a set $\mathcal{M}$ that consists of all monomials
in a given $\ZZ^r$-degree. Moreover, if we assume that $V$
has a parametrization by sparse polynomials then we
would use a specialized combinatorial analysis
to predict a set
$\mathcal{M}$ that works. A suitable choice of $\mathcal{M}$ can improve the numerical accuracy of the computations dramatically.

In addition to choosing the set of monomials $\mathcal{M}$, we face another problem: how to represent
 $I_\Omega \cap R_\mathcal{M}$? Computing a basis for the kernel of $U_\mathcal{M}(\Omega)$
 yields a set of generators for $I_\Omega \cap R_\mathcal{M}$. But which basis to use and
 how to compute it? For instance, the
 right-singular vectors of $U_\mathcal{M}(\Omega)$ with singular
value zero yield an \emph{orthonormal basis} of $I_\Omega \cap R_\mathcal{M}$.
But in applications one often meets ideals $I$ that have sparse generators.
This holds in Section~\ref{sec:vardata}.

\begin{example} \label{example_equations}
Suppose that we obtain a list of $20$ quadrics in nine variables as the result of computing the kernel of a Vandermonde matrix and each quadric looks
something like this:
\begin{tiny}
\begin{align*}
 &-0.037x_1^2 - 0.043x_1x_2 - 0.011x_1x_3 + 0.041x_1x_4 - 0.192x_1x_5 + 0.034x_1x_6 + 0.031x_1x_7 + 0.027x_1x_8+ 0.271x_1x_9 + 0.089x_2^2- 0.009x_2x_3\\
 &  + 0.192x_2x_4 + 0.041x_2x_5 + 0.044x_2x_6 - 0.027x_2x_7 + 0.031x_2x_8- 0.048x_2x_9 - 0.056x_3^2 - 0.034x_3x_4 - 0.044x_3x_5 + 0.041x_3x_6\\
 & - 0.271x_3x_7 + 0.048x_3x_8 + 0.031x_3x_9- 0.183x_4^2 - 0.043x_4x_5 - 0.011x_4x_6 + 0.039x_4x_7 + 0.004x_4x_8 + 0.019x_4x_9 - 0.057x_5^2\\
 & - 0.009x_5x_6 - 0.004x_5x_7 + 0.039x_5x_8 - 0.35x_5x_9 - 0.202x_6^2 - 0.019x_6x_7 + 0.35x_6x_8 + 0.039x_6x_9 - 0.188x_7^2 - 0.043x_7x_8 - 0.011x_7x_9\\
 & - 0.062x_8^2 - 0.009x_8x_9 - 0.207x_9^2 + 0.35x_1 + 0.019x_2 - 0.004x_3 - 0.048x_4 - 0.271x_5 + 0.027x_6 - 0.044x_7 + 0.034x_8 + 0.192x_9 + 0.302
\end{align*}
\end{tiny}
This is the first element in an orthonormal basis for $I_\Omega \cap R_{\leq 2}$,
where $\Omega$ is a sample drawn from a certain variety $V$ in $\RR^9$.
From such a basis, it is very hard to guess what $V$ might~be.

It turns out that $V $ is $ {\rm SO}(3)$, the group of rotations in $3$-space.
After renaming the nine variables, we find the $20$-dimensional space of quadrics
mentioned in Example \ref{ex:rotation}.
However, the quadrics seen in  (\ref{eq:nineSO3}) are much nicer.
They are sparse and easy to interpret.
\end{example}

For this reason we aim to compute \emph{sparse} bases of multivariate Vandermonde matrices.
There is a trade-off between obtaining sparse basis vectors and stability of the computations.
We shall discuss this issue in the next subsection.  See
Table~\ref{table_methods} for a brief summary.

\subsection{Numerical Linear Algebra}

Computing kernels of matrices of type $U_{\mathcal{M}}(\Omega)$ is a problem in numerical linear algebra. One scenario where the methodology has been developed and
proven to work well is the Generalized Principal Component Analysis
of Ma {\it et al.}~\cite{derksen}, where $V$ is a finite union of linear subspaces in $\RR^n$. For classical Vandermonde matrices, the Bjoerck-Pereyra algorithm~\cite{Bjoerck} accurately computes a LU-decomposition of the Vandermonde matrix; see \cite[Section~22]{higham}. This decomposition may then be used to compute the kernel. A generalization of this for multivariate Vandermonde matrices of the form $U_{\leq d}(\Omega)$ is given in \cite[Theorem 4.4]{Olver}. To date such a decomposition for $U_{\mathcal M}(\Omega)$ is missing for other subsets of monomials $\mathcal M$. Furthermore, \cite[Theorem 4.4]{Olver} assumes that the multivariate Vandermonde matrix is square and invertible, but this is never the case in our situation.

In the literature on numerical algebraic geometry, it is standard to represent varieties
by point samples, and there are several approaches for learning varieties, and even schemes,
from such numerical data. See e.g.~\cite{DH, GHPS} and the references therein.
From the perspective of commutative algebra, our
interpolation problem was studied in e.g.~\cite{BM, Mus}.

We developed and implemented three methods based on classical numerical linear algebra:
\begin{enumerate}
\item    via the R from a QR-decomposition, 
\vspace{-0.1in}
\item   via a singular value decomposition (SVD), or
\vspace{-0.1in}
\item via the reduced row echelon form (RREF) of $U_{\mathcal M}(\Omega)$.
\end{enumerate}
The goal is to compute a (preferably sparse) basis for the kernel of
$\,U_{\mathcal M}(\Omega)$, with $N = |\mathcal{M}|$.
 All three methods are implemented in our software.
Their descriptions are given below.
\begin{table}[h!]
      \centering
\begin{tabular}{lll}\hline
 ~&QR & slightly less accurate and fast than SVD, yields some sparse basis vectors.\cr
  ~&SVD & accurate, fast, but returns orthonormal and hence dense basis.\cr
 ~&RREF & no accuracy guarantees, not as fast as the others, gives a sparse basis.\\\hline
\end{tabular}
\caption{The three methods for computing the kernel of
the Vandermonde matrix $U_\mathcal{M}(\Omega)$.}
\label{table_methods}
\end{table}

\begin{small}
\smallskip

\begin{algorithm}[H]
\KwIn{A multivariate Vandermonde matrix $U\in\mathbb{R}^{m\times N}$ and a tolerance value $\tau\geq 0$.}
\KwOut{A basis for the kernel of $U$.}
Compute the QR-decomposition $U = QR$, where $Q$ is orthogonal and $R$ is upper triangular\;
Put $I = \{i : 1\leq i\leq N, \vert R_{ii} \vert <\tau\}$, $J=[N]\backslash I$, $R' = R^{[m]\times J}$ and $\mathcal{B} = \emptyset$\;
\For{$i\in I$}{Initialize $a\in\mathbb{R}^N$, $a=(a_1,\ldots,a_N)$ and put $a_i=1$\;
Solve $R'y=R_i$ for $y$, where $R_i$ is the $i$-th column of $R$.\;
Put $(a_1,\ldots,a_{i-1},a_{i+1},\ldots,a_N) = y$\;
Update $\mathcal{B} \leftarrow \mathcal{B}\cup \{a\}$\;
}
Return $\mathcal{B}$.
\caption{with\_qr\label{qr}}
\end{algorithm}

\smallskip

\begin{algorithm}[H]
\KwIn{A multivariate Vandermonde matrix $U\in\mathbb{R}^{m\times N}$ and a tolerance value $\tau\geq 0$.}
\KwOut{A basis for the kernel of $U$.}
Compute the singular value decomposition $U = X\Sigma Y$, where $\Sigma = \mathrm{diag}(\sigma_1,\ldots,\sigma_N)$.\;
Let $k:=\#\{1\leq i\leq N: \sigma_i<\tau\}$\;
Return the last $k$ columns of $Y$\;
\caption{with\_svd\label{svd}}
\end{algorithm}
 \medskip

\smallskip

\begin{algorithm}[H]
\KwIn{A multivariate Vandermonde matrix $U\in\mathbb{R}^{m\times N}$ and a tolerance value $\tau\geq 0$.}
\KwOut{A basis for the kernel of $U$.}
Compute the reduced row-echelon form $A$ of $U$\;
Put $I = \{i : 1\leq i\leq m, \Vert A_{i} \Vert > \sqrt{N}\tau\}$, where $A_i$ is the $i$-th row of $A$\;
Put $B:=A^{I\times [N]}$, $k:=\#I$ and initialize $\mathcal{B}=\emptyset$\;
For $1\leq i\leq k$ let $j_i$ be the position of the first entry in the $i$-th row of $B$ that has absolute value larger than $\tau$ and put $J:=[N]\backslash\{j_1,\ldots,j_k\}$\;
\For{$j\in J$}{Put $J':=\{1\leq i\leq N : i<j\}$\;
Initialize $a\in\mathbb{R}^N$, $a=(a_1,\ldots,a_N)$ and put $a_j = 1$ and $a_i = 0$ for $i\neq j$.\;
\For{$i\in J'$}{ $a_{i} = -B_{i,j}$\; Update $\mathcal{B} \leftarrow \mathcal{B}\cup \{a\}$\;}
}
Return $\mathcal{B}$.
\caption{with\_rref\label{rref}}
\end{algorithm}
\end{small}

\medskip

Each of these three methods has its upsides and downsides. These are summarized
in Table \ref{table_methods}.
The algorithms require a tolerance $\tau\geq 0$ as input. This tolerance value determines the {\em numerical rank} of the matrix. Let $\sigma_1 \geq \cdots \geq \sigma_{\min\{m,N\}}$ be the ordered singular values of the $m\times N$ matrix $U$.
As in the beginning of Subsection \ref{sub:dimension},
the numerical rank of $U$ is
\begin{equation}
\label{eq:numericalrank}
 r(U,\tau) \,\,:=\,\,\# \bigl\{\,i\,\, |\,\, \sigma_i \geq  \tau\,\bigr\}.
\end{equation}
Using the criterion in  \cite[\S 3.5.1]{Demmel}, we can set
$\tau = \varepsilon\, \sigma_1\, \max\{m,N\}$, where~$\epsilon$ is the machine precision. The rationale behind this choice is \cite[Corollary 5.1]{Demmel}, which says that the round-off error in the $\sigma_i$ is bounded by $\Vert E\Vert$, where $\Vert\cdot \Vert$ is the spectral norm and $U+E$ is the matrix whose singular values were computed. For backward stable algorithms we may use the bound $\Vert E\Vert = \mathcal O(\varepsilon) \sigma_1$. On the other hand,  our experiments suggest that an appropriate value for $\tau$ is given by $\frac{1}{2}(\sigma_i + \sigma_{i+1})$, for which the jump from $\log_{10}(\sigma_i)$ to $\log_{10}(\sigma_{i+1})$ is significantly large. This choice is particularly useful for noisy data (as seen in Subsection~\ref{sub:cyclo-octane}).
In case of noise the first definition of~$\tau$ will likely fail to detect the true rank
 of $U_{\leq d}(\Omega)$.
  The reason for this lies in the numerics of Vandermonde matrices, discussed below.

We apply all of the aforementioned to the multivariate Vandermonde matrix
$U_\mathcal{M}(\Omega)$,
for any finite set  $\mathcal{M}$  in $ R$ that is linearly independent.
We thus arrive at the following algorithm.

\begin{small}
\medskip
\begin{algorithm}[H]
\KwIn{A sample of points $\Omega =
  \{u^{(1)},u^{(2)}, \ldots, u^{(m)}\} \subset \mathbb{R}^n$,
  a finite set $\mathcal M$  of monomials  in $n$ variables,
  and a tolerance value $\tau>0$.
}
\KwOut{A basis $\mathcal{B}$ for the kernel of $U_{\mathcal M}(\Omega)$\;}
Construct the multivariate Vandermonde matrix $U_{\mathcal M}(\Omega)$\;
Compute a basis $\mathcal{B}$ for the kernel of $U_{\mathcal M}(\Omega)$ using
Algorithm \ref{qr}, \ref{svd} or \ref{rref}\;
Return $\mathcal{B}$\;
\caption{FindEquations\label{alg_equations}}
\end{algorithm}
\medskip
\end{small}

\begin{remark}
Different sets of quadrics can be obtained by applying
Algorithm \ref{alg_equations} to a set $\Omega$ of $200$ points
sampled uniformly from the  group $\mathrm{SO}(3)$.
The dense equations in Example \ref{example_equations} are obtained
using Algorithm~\ref{svd} (SVD) in Step 4.  The more desirable sparse equations from~(\ref{eq:nineSO3})
are found when using Algorithm~\ref{qr} (with QR).
In both cases the tolerance was set to be $\tau \approx 4 \cdot 10^{-14}\, \sigma_1\,$, where $\sigma_1$ is the largest singular value of the Vandermonde matrix~$U_{\leq 2}(\Omega)$.
 \end{remark}

Running Algorithm \ref{alg_equations} for a few good choices of $\mathcal{M}$
often leads to an initial list of non-zero polynomials that lie in
$I_\Omega$ and also in $I_V$.
Those polynomials can then be used to infer
an upper bound on the dimension and other information about $V$.
This is explained in Section~\ref{sec:using}.
Of course, if we are lucky, we obtain a generating
set for $I_V$ after a few iterations.

If $m$ is not too large and the coordinates of the points $u^{(i)}$
are rational, then it can be preferable to compute the kernel of $U_{\mathcal{M}}(\Omega)$
symbolically. Gr\"obner-based interpolation methods, such as
the {\em Buchberger-M\"oller algorithm} \cite{BM},
have the flexibility to select $\mathcal{M}$ dynamically. With this, they directly
compute the generators for the ideal $I_\Omega$,
rather than the user having to worry about  the matrices
$U_{\leq d}(\Omega)$ for a sequence of degrees $d$.
In short, users should keep symbolic methods in the
back of their minds when contemplating Algorithm \ref{alg_equations}.

In the remainder of this section, we discuss numerical issues associated
with Algorithm~\ref{alg_equations}. The key step is computing the kernel of the multivariate Vandermonde matrix $U_{\mathcal M}(\Omega)$. As illustrated in (\ref{Vandermonde}) for
samples $\Omega $ on the line $(n=1)$,
 and $\mathcal M$ being all monomials up to a fixed degree, this matrix is a \emph{Vandermonde matrix}. It is conventional wisdom
 that Vandermonde matrices are severely ill-conditioned \cite{Pan}. Consequently, numerical linear algebra solvers are expected to perform poorly when attempting to compute the kernel of $U_d(\Omega)$.

One way to circumvent this problem is to use a set of \emph{orthogonal polynomials} for $\mathcal{M}$. Then, for large sample sizes $m$, two distinct columns of $U_{\mathcal{M}}(\Omega)$ are approximately orthogonal, implying that $U_{\mathcal{M}}(\Omega)$ is well-conditioned. This is because the inner product between the columns associated to $f_1,f_2\in\mathcal{M}$ is approximately the integral of $f_1\cdot f_2$ over $\mathbb{R}^n$. However, a sparse representation in orthogonal polynomials does not yield a sparse representation in the monomial basis. Hence, to get sparse polynomials in the monomials basis from $U_{\mathcal{M}}(\Omega)$, we must employ other methods than the ones presented here. For instance, techniques from compressed sensing may help to compute sparse representations in the monomial basis.

We are optimistic that a numerically-reliable algorithm for computing the kernel of matrices $U_{\leq d}(\Omega)$ exists. The Bjoerck-Pereyra algorithm~\cite{Bjoerck} solves linear equations $Ua=b$ for
an $n \times n$ Vandermonde matrix $U$.
There is a theoretical guarantee that the computed solution~$\hat{a}$ satisfies $\vert a-\hat{a}\vert \leq 7 n^5 \epsilon + \mathcal O(n^4\epsilon^2)$; see \cite[Corollary 22.5]{higham}. Hence, $\hat{a}$ is highly accurate -- despite $U$ being ill-conditioned. This is confirmed by the experiment mentioned in the beginning of \cite[Section 22.3]{higham}, where a linear system with $\kappa(U) \sim 10^9$ is solved with a relative error of~$5\epsilon$. We suspect that a Bjoerck-Pereyra-like algorithm together with a thorough structured-perturbation analysis for multivariate Vandermonde matrices would equip us with an accurate algorithm for finding equations. For the present article,
we stick with the three methods  above, while bearing in mind the difficulties that ill-posedness
can cause.

\section{Learning from Equations}\label{sec:using}

At this point we assume that the methods in the previous two sections
have been applied. This means that we have an estimate $d$ of what the dimension
of $V$ might be, and we know a set $\mathcal{P}$ of polynomials that vanish on
the finite sample $\Omega \subset \RR^n$.
We assume that the sample size $m$ is large enough so that
the polynomials in $\mathcal{P}$ do in fact vanish on $V$.
We now use $\mathcal{P}$ as our input. Perhaps the unknown variety $V$ is
 one of the objects seen in  Subsection \ref{subsec:variety}.

\subsection{Computational Algebraic Geometry}

A finite set of polynomials  $\mathcal{P}$ in $\QQ[x_1,\ldots,x_n]$
is the typical input for algebraic geometry software.
Traditionally, symbolic packages like
{\tt Macaulay2}, {\tt Singular} and {\tt CoCoA}
were used to study $\mathcal{P}$. Buchberger's
Gr\"obner basis algorithm is the workhorse underlying
this approach. More recently, numerical algebraic geometry has emerged, offering
 lots of promise for innovative and accurate methods in data analysis.
We refer to the textbook \cite{BertiniBook}, which centers around the excellent software {\tt Bertini}.
Next to using \texttt{Bertini}, we also employ the \texttt{Julia} package \texttt{HomotopyContinuation.jl} \cite{BT}. Both symbolic and numerical methods are valuable for data analysis.
The questions we ask in this subsection can be answered with~either.

In what follows we assume that the unknown variety $V$ is equal to
the zero set of the input polynomials $\mathcal{P}$.
We seek to answer the following questions
over the complex numbers:

\begin{enumerate}
\item What is the dimension of $V$? \vspace{-0.2cm}
\item What is the degree of $V$? \vspace{-0.21cm}
\item Find the irreducible components of $V$
and determine their dimensions and degrees.
\end{enumerate}

Here is an example
that illustrates the workflow we imagine for analyzing samples $\Omega$.

\begin{example} \label{ex:hankel44}
The variety of Hankel matrices of size $4 \times 4$
and rank $2$ has the parametrization
$$ {\footnotesize
\begin{bmatrix}
 a & b & c & x \\
 b & c & x & d \\
 c & x & d & e \\
 x & d & e & f
 \end{bmatrix} \quad = \quad
\begin{bmatrix}
s_1^3 & s_2^3 \\
s_1^2 t_1 & s_2^2 t_2 \\
s_1 t_1^2 & s_2 t_2^2 \\
t_1^3 & t_2^3
\end{bmatrix}
\begin{bmatrix}
s_1^3 & s_1^2 t_1 & s_1 t_1^2 & t_1^3 \\
s_2^3 & s_2^2 t_2 & s_2 t_2^2 & t_2^3
\end{bmatrix}.
} $$
Suppose that an adversary constructs a dataset $\Omega$
of size $m=500$ by the following process.
He picks random integers $s_i$ and $t_j$, computes the $4 \times 4$-Hankel
matrix, and then deletes the antidiagonal coordinate $x$.
For the remaining six coordinates he fixes some random ordering, such as
 $(c,f,b,e,a,d)$. Using this ordering, he lists the $500$ points.
 This is our input $\Omega \subset \RR^6$.

 We now run Algorithm \ref{alg_equations} for the
  $m \times 210   $-matrix $U_{\leq 4}(\Omega)$.
The output of this computation is the following pair of quartics
which vanishes on the variety $V \subset \RR^6$ that is described above:
\begin{equation}
\label{eq:Ptwo}
{\footnotesize
\begin{matrix}
\mathcal{P} &= & \bigl\{\, acf^2+ad^2f-2ade^2-b^2f^2+2bd^2e-c^2df+c^2e^2-cd^3,  \\
&& \quad \,\,\, a^2 df-a^2e^2+ac^2f-acd^2-2b^2cf+b^2d^2+2bc^2e-c^3d \,\bigr\}.
\end{matrix} }
\end{equation}
Not knowing the true variety, we optimistically believe that the zero set of $\mathcal{P}$ is equal to $V$.
This would mean that $V$ is a complete intersection, so it has
codimension $2$ and degree $16$.

At this point, we may decide to compute a primary decomposition of $\langle \mathcal{P} \rangle$.
We then find that there are two components of codimension $2$, one
of degree $3$ and the other of degree $10$.
Since $3 + 10 \not= 16$, we learn that $\langle \mathcal{P} \rangle$
is not a radical ideal. In fact, the degree $3$ component appears with multiplicity $2$.
Being intrigued, we now return to computing equations from~$\Omega$.

From the kernel of the $m \times 252$-matrix $U_5(\Omega)$, we find two
new quintics in $I_\Omega$. These only reduce the degree to $3+10= 13$.
Finally, the kernel of the $m \times 452$-matrix $U_6(\Omega)$ suffices.
The ideal $I_V$ is generated by $2$ quartics,
$2$ quintics and $4$ sextics.
The mystery variety $V\subset \RR^6$ has the same dimension and degree
as  the rank $2$ Hankel variety in $\RR^7$ whose projection it~is.
\end{example}

Our three questions boil down  to solving a system   $\mathcal{P}$
of polynomial equations.   Both symbolic and numerical techniques can be used for that task.
Samples $\Omega$ seen in applications
are often large, are represented by
floating numbers, and have errors and outliers. In those cases, we use {\em Numerical Algebraic Geometry} \cite{BertiniBook, BT}. For instance, in Example \ref{ex:hankel44} we intersect~(\ref{eq:Ptwo}) with a linear space of dimension $2$. This results in $16$ isolated solutions. Further numerical analysis in step 3 reveals the desired irreducible component of degree $10$.

\smallskip

In the numerical approach to answering the three questions, one proceeds as follows:
\begin{enumerate}
\item We add $s$ random (affine-)linear equations to $\mathcal{P}$
and we solve the resulting system in~$\CC^n$. If there are no solutions,
then ${\rm dim}(V) < s$. If the solutions
are not isolated, then~${\rm dim}(V) > s$.
Otherwise, there are finitely many solutions, and
${\rm dim}(V) = s$.
\item The degree of $V$ is the finite number of
solutions found in step 1.
\item  Using {\em monodromy loops} (cf.~\cite{BertiniBook}),
we can identify the intersection of a linear space $L$ with any
irreducible component of $V_{\mathbb{C}}$ whose
codimension equals ${\rm dim}(L)$.
\end{enumerate}

The dimension diagrams from Section \ref{sec:dimension}
can be used to guess a suitable range of values for the
parameter $s$ in step 1. However, if we have equations at hand, it is better to determine the dimension $s$ as follows. Let $\mathcal{P} = \{f_1,\ldots,f_k\}$ and $u$ be any data point in $\Omega$. Then, we choose the $s$ from step 1 as the corank of the Jacobian matrix of $f=(f_1,\ldots,f_k)$ at $u$; i.e,
  \begin{equation}\label{how_to_set_s}
   s := \mathrm{dim}\, \mathrm{ker}\, Jf(u).
\end{equation}
Note that $s=\mathrm{dim}\, V(\mathcal{P})$ as long as $u$ is not a singular point of $V(\mathcal{P})$. In this case, $s$ provides an upper bound for the true dimension of $V$.
 That is why it is important in step~3 to use
higher-dimensional linear spaces $L$ to detect lower-dimensional components of $V(\mathcal{P})$.

 \begin{example}
 Take $m=n=3$ in Example \ref{ex:rankone}.
 Let $\mathcal{P}$ consist of the four
 $2 \times 2$-minors that contain the upper-left matrix entry $x_{11}$.
 The ideal  $\langle \mathcal{P} \rangle$ has codimension $3$ and degree~$2$.
 Its top-dimensional components are
 $\langle x_{11}, x_{12}, x_{13} \rangle$ and
 $ \langle x_{11}, x_{21}, x_{31} \rangle$. However, our
 true model~$V$ has codimension $4$ and degree $6$: it is defined
 by all nine $2 \times 2$-minors.
 Note that $\langle \mathcal{P} \rangle$ is not radical.
 It also has an embedded prime of codimension $5$, namely
  $\langle x_{11}, x_{12}, x_{13} ,  x_{21}, x_{31}  \rangle$.
  \end{example}

\subsection{Real Degree and Volume}

The discussion in the previous subsection was about the complex points
of the variety $V$. The geometric quantity ${\rm deg}(V)$
records a measurement over $\CC$. It is insensitive to the
geometry of the real points of $V$. That perspective does not distinguish between
$\mathcal{P} = \{x^2+y^2-1\}$ and $\mathcal{P} = \{x^2+y^2+1\}$.
That distinction is seen through the lens of  {\em real algebraic geometry}.

In this subsection we study metric properties of
 a real projective variety $V \subset \PP^{n}_\RR$.
We explain how to estimate the {\em volume} of $V$. Up to a constant depending on $d=\mathrm{dim}\, V$,
this volume equals the {\em real degree} ${\rm deg}_\RR(V)$, by which we mean
the expected number of real intersection points with a linear subspace
of codimension ${\rm dim}(V)$; see Theorem \ref{kinematic} below.

To derive these quantities, we use \emph{Poincar\'e's kinematic formula} \cite[Theorem~3.8]{Howard}.
 For this we need some notation. By \cite{Leichtweiss} there is a unique orthogonally invariant measure $\mu$ on~$\mathbb{P}_\RR^n$ up to scaling. We choose the scaling in a way compatible
 with the unit sphere~$\mathbb{S}^{n}$:
 $$ \mu(\mathbb{P}_\RR^n)\,\,= \frac{1}{2}\mathrm{vol}(\mathbb{S}^{n})
\,\, =\,\, \frac{\pi^{\frac{n+1}{2}}}{\Gamma(\frac{n+1}{2})}.$$
 This makes sense  because $\mathbb{P}_\RR^n$ is doubly covered by~$\mathbb{S}^n$.
 The $n$-dimensional volume $\mu$ induces a~$d$-dimensional
 measure of volume on $\PP^n_\RR$ for any $d = 1,2,\ldots,n-1$.
 We use that measure for~$d = {\rm dim}(V)$ to define the
 volume of our real projective variety as
     $\mathrm{vol}(V):=\mu(V)$.

  Let $\mathrm{Gr}(k,\mathbb{P}_\RR^n)$ denote the Grassmannian
     of $k$-dimensional linear spaces in $\mathbb{P}_\RR^n$.
     This is a real manifold of dimension $(n-k)(k+1)$. Because of the Pl\"ucker embedding it is also
     a projective variety. We saw this
for $k=1$ in Example  \ref{ex:pfaffians},  but  we will not use it here. Again by~\cite{Leichtweiss}, there is a unique orthogonally invariant measure $\nu$
on $\mathrm{Gr}(k,\mathbb{P}_\RR^n)$ up to scaling. We choose the scaling $\nu(\mathrm{Gr}(k,\mathbb{P}_\RR^n))=1$.
This defines the \emph{uniform probability distribution} on the Grassmannian.
   Poincar\'e's Formula    \cite[Theorem 3.8]{Howard} states:

\begin{theorem}[Kinematic formula in projective space]\label{kinematic}
Let $V$ be a smooth projective variety  of codimension $k=n-d$
in $\PP^n_\RR$. Then its volume is the volume of $\,\PP^d_\RR$ times the real degree:
$$\mathrm{vol}(V) \,=\, \frac{\pi^{\frac{d+1}{2}}}{\Gamma(\frac{d+1}{2})} \cdot {\rm deg}_\RR(V)
\quad \text{where} \quad
{\rm deg}_\RR(V) \,=
\, \int_{L\in\mathrm{Gr}(k,\mathbb{P}_\RR^n)} \, \#(L\cap V) \, \mathrm{d} \nu.$$
\end{theorem}

Note that in case of $V$ being a linear space of dimension $d$, we have $\#(L\cap V) = 1$ for all $L\in\mathrm{Gr}(n-d,\mathbb{P}_\RR^n)$. Hence, $\mathrm{vol}(V) = \mathrm{vol}(\mathbb{P}_\mathbb{R}^d)$, which verifies the theorem in this instance.

The theorem suggests an algorithm. Namely, we
sample linear spaces $L_1,L_2,\ldots,L_N$ independently and uniformly at random, and compute
the number $r(i)$ of real points in $V\cap L_i$ for each $i$.
This can be done symbolically (using Gr\"obner bases)
or numerically (using homotopy continuation).
We obtain the following estimator for ${\rm vol}(V)$:
$$
\widehat{ {\rm vol}}(V) \,\,=\,\,
\frac{\pi^{\frac{d+1}{2}}}{\Gamma(\frac{d+1}{2})} \cdot\frac{1}{N}\sum_{i=1}^N r(i).
$$

We can sample uniformly from $\mathrm{Gr}(k,\mathbb{P}_\RR^n)$ by using the following lemma:

\begin{lemma}
\label{lem:poincare}
Let $A$ be a
random $(k{+}1) \times (n {+} 1)$ matrix with independent standard Gaussian entries. The
row span of $A$ follows the uniform distribution  on
the Grassmannian $\mathrm{Gr}(k,\mathbb{P}_\RR^n)$.
\end{lemma}

\begin{proof}
The distribution of the row space of $A$ is orthogonally invariant. Since the
orthogonally invariant probability measure on  $\mathrm{Gr}(k,\mathbb{P}_\RR^n)$
is unique,  the two distributions agree.
\end{proof}

\begin{example} \rm
Let $n=2$, $k=1$, and let $V$ be the {\em Trott curve} in $\PP^2_\RR$.
The area of the projective plane $\PP^2_\RR$ is half of the surface area
of the unit circle:
$ \mu(\PP^1_\RR) \,= \, \frac{1}{2} \cdot {\rm vol}(\mathbb{S}^1) \,=\, \pi$.
 The real degree of $V$ is computed with the method suggested in
 Lemma \ref{lem:poincare}:
${\rm deg}_\RR (V) \,=\,1.88364. $
We estimate the length of the Trott curve to be the product of these two numbers: $\, 5.91763$. Note that $5.91763$ does \emph{not} estimate the length of the affine curve depicted in Figure \ref{fig:trott2},
but it is the length of the projective curve  defined by the
homogenization of the polynomial (\ref{eq:trotteqn}).
\end{example}

\begin{remark}
Our discussion in this subsection focused on real projective varieties.
For affine varieties $V\subset \mathbb{R}^n$ there is a
formula similar to Theorem \ref{kinematic}. By \cite[(14.70)]{Santalo},
$$\mathrm{vol}(V) \,\,=\,\, \, \frac{O_{n-d}\cdots O_1}{O_n\cdots O_{d+1}}\cdot \
\int_{L \cap V\neq \emptyset}\#(V\cap L) \,\mathrm{d} L,\qquad d=\mathrm{dim}\,V,$$
where $\mathrm{d} L$ is the density of affine $(n-d)$-planes in $\mathbb{R}^n$ from \cite[Section 12.2]{Santalo}, $\mathrm{vol}(\cdot)$ is Lebesgue measure in~$\mathbb{R}^n$ and $O_m:=\mathrm{vol}\,(\mathbb{S}^m)$. The problem with using this formula is that in general we do not know how to sample from the density $\mathrm{d}L$ given $L\cap V \neq \emptyset$. The reason is that this distribution depends on $\mathrm{vol}(V)$--which we were trying to compute in the first place.

Suppose that the variety $V$ is the image of a parameter space over which integration is easy.
This holds for $V = \mathrm{SO}(3)$, by (\ref{eq:SO3para}). For such cases, here is
an alternative approach for computing the volume: pull back the volume form on $V$ to
the parameter space and integrate it there. This can be done either numerically or --if possible-- symbolically.
Note that this method is not only applicable to smooth varieties, but to any differentiable manifold.
\end{remark}

\section{Software and Experiments}\label{sec:experiments_and_data}
In this section, we demonstrate how the methods from previous sections work in practice. The implementations are available in our \texttt{Julia} package \texttt{LearningAlgebraicVarieties}.
We offer a step-by-step tutorial.
To install our software, start a \texttt{Julia} session and type
\begin{lstlisting}
Pkg.clone("https://github.com/PBrdng/LearningAlgebraicVarieties.git")
\end{lstlisting}
After the installation, the next command is
\begin{lstlisting}
using LearningAlgebraicVarieties
\end{lstlisting}
This command loads all the functions into the current session. Our package accepts a dataset~$\Omega$ as a matrix whose~\emph{columns} are the data points $u^{(1)},u^{(2)},\ldots,u^{(m)}$ in $\RR^n$.

To use the numerical algebraic geometry software \texttt{Bertini}, we must first download it from \url{https://bertini.nd.edu/download.html}. The \texttt{Julia} wrapper for \texttt{Bertini} is installed by
\begin{lstlisting}
Pkg.clone("https://github.com/PBrdng/Bertini.jl.git")
\end{lstlisting}
The code \texttt{HomotopyContinuation.jl} accepts input from
the polynomial algebra package
 \texttt{MultivariatePolynomials.jl}\footnote{\url{https://github.com/JuliaAlgebra/MultivariatePolynomials.jl}}.
The  former is described in \cite{BT} and it is installed using
\begin{lstlisting}
Pkg.add("HomotopyContinuation")
\end{lstlisting}

We apply our package to three datasets.
The first comes from the group $\mathrm{SO}(3)$,
the second from the projective variety $V$ of $2\times3$-matrices $(x_{ij})$ of rank $1$,
and the third from the conformation space of \emph{cyclo-octane}.

In the first two cases, we draw the samples ourselves. The introduction of~\cite{Diaconis} mentions algorithms to sample from compact groups. However, for the sake of simplicity we use the following algorithm for sampling from $\mathrm{SO}(3)$. We use \texttt{Julia}'s \lstinline{qr()}-command to compute the QR-decomposition of a random real $3\times 3$ matrix with independent standard Gaussian entries and take the $Q$ of that decomposition. If the computation is such that the diagonal entries of $R$ are all positive
then, by \cite[Theorem~1]{Mezzadri}, the matrix $Q$ is
 uniformly distributed in~$\mathrm{O}(3)$. However, in our
 case, $Q\in\mathrm{SO}(3)$ and we do not know its distribution.

Our sample from the {\em Segre variety} $V = \PP^1_\RR \times \PP^2_\RR$
  in $\PP^5_\RR$   is drawn by  independently sampling two standard Gaussian
  matrices of format $2\times 1$ and~$1\times 3$ and multiplying them.
   This procedure yields the uniform distribution on  $V$ because
   the Segre embedding is an isometry under the
   Fubini-Study metrics on $\PP^1_\RR, \PP^2_\RR$ and $\PP^5_\RR$.
    The third sample, which is $6040$ points from the conformation space of cyclo-octane, is taken from
    Adams {\it et al.}~\cite[\S. 6.3]{JavaplexTutorial}.

We provide the samples used in the subsequent experiments in the \texttt{JLD}\footnote{\url{https://github.com/JuliaIO/JLD.jl}} data format. After having installed the \texttt{JLD} package in \texttt{Julia} (\lstinline{Pkg.add("JLD")}),
 load the datasets by~typing
\begin{lstlisting}
import JLD: load
s = string(Pkg.dir("LearningAlgebraicVarieties"),"/datasets.jld")
datasets = load(s)
\end{lstlisting}

\subsection{Dataset 1: a sample from the rotation group $\mathrm{SO}(3)$}
The group $\mathrm{SO}(3)$ is a variety in the space of $3 \times 3$-matrices. It is defined by
the polynomial equations in  Example~\ref{ex:rotation}.
 A dataset containing $887$ points from $\mathrm{SO}(3)$ is loaded by typing
\begin{lstlisting}
data = datasets["SO(3)"]\end{lstlisting}
Now the current session should contain a variable \lstinline{data} that is a $9\times 887$ matrix.
We produce the dimension diagrams by typing
\begin{lstlisting}
DimensionDiagrams(data, false, methods=[:CorrSum,:PHCurve])
\end{lstlisting}
In this command, \lstinline{data} is our dataset, the Boolean value is \lstinline{true} if we suspect our variety is projective and \lstinline{false} otherwise, and \lstinline{methods} is any of
the dimension estimates \lstinline{:CorrSum}, \lstinline{:BoxCounting} \lstinline{:PHCurve}, \lstinline{:NPCA}, \lstinline{:MLE}, and \lstinline{:ANOVA}.
We can leave this unspecified and type
\begin{lstlisting}
DimensionDiagrams(data, false)
\end{lstlisting}
This command plots  all six dimension diagrams. Both outputs are shown in Figure~\ref{DDSO3}.

\begin{figure}[h!]
\begin{center}
\includegraphics[scale=0.33]{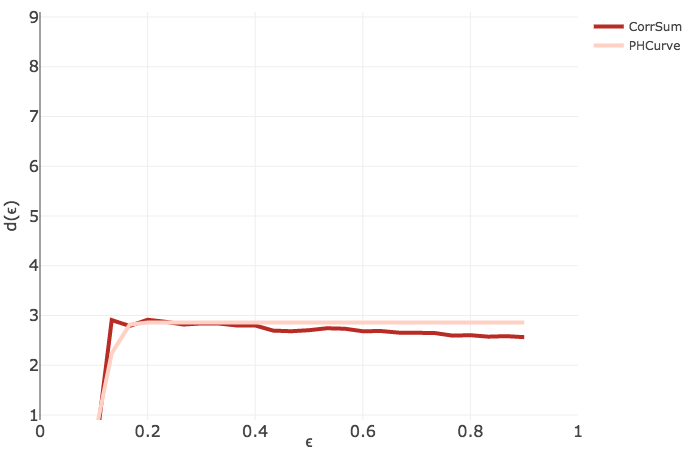}
\includegraphics[scale=0.33]{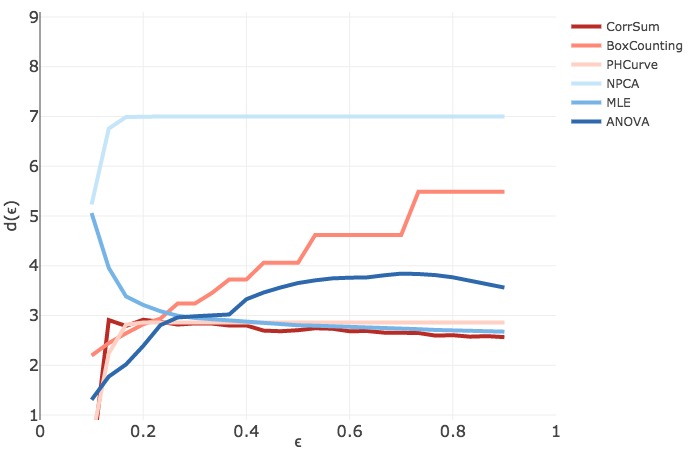}
\caption{Dimension diagrams for $887$ points in ${\rm SO}(3)$.
 The right picture shows all six diagrams described in Subsection~\ref{sub:dimension}.
 The left picture shows correlation sum and persistent homology curve dimension estimates. }\label{DDSO3}
\end{center}
\end{figure}

 Three estimates are close to $3$, so we correctly guess the true dimension of ${\rm SO}(3)$.
In our experiments we found that NPCA and Box Counting Dimension often overestimate.

We proceed by finding polynomials that vanish on the sample.
 The command we use is
\begin{lstlisting}
FindEquations(data, method, d, homogeneous_equations)
\end{lstlisting}
where \lstinline{method} is one of \texttt{:with\textunderscore svd}, \texttt{:with\textunderscore qr}, \texttt{:with\textunderscore rref}. The degree~\lstinline{d} refers to the polynomials in $R$ we are looking for.
 If \texttt{homogeneous\textunderscore equations} is set to \lstinline{false}, then  we search in $R_{\leq d}$.
   If we look for a projective variety then
    we set it to \lstinline{true}, and $R_d$ is used.
For our sample from ${\rm SO}(3)$ we use the \lstinline{false} option.
Our sample size $m=887$ is large enough to determine equations up to $d=4$.
The following results are  found by the various methods:

\medskip

\begin{center}
\begin{tabular}{cccc}
$d$ & method & number of linearly independent equations  \\ \hline
$1$ & SVD & 0 \\
$2$ & SVD & 20 \\
$2$ & QR & 20 \\
$2$ & RREF & 20\\
$3$ & SVD & 136\\
$4$ & SVD & 550
\end{tabular}
\end{center}
The correctness of these numbers can be verified by
computing  ({\it e.g.~}using \texttt{Macaulay2})
the affine Hilbert function  \cite[\S 9.3]{CLO} of the ideal
with the generators in Example \ref{ex:rotation}.
If we type
\begin{lstlisting}
f = FindEquations(data, :with_qr, 2, false)
\end{lstlisting}
then we get a list of 20 polynomials that vanish on the sample.

The output is often difficult to interpret, so it can be desirable to round the coefficients:
 \begin{lstlisting}
 round.(f)
\end{lstlisting}
The precision can be specified, the default being to the nearest integer. We obtain the output
  \[
  \begin{array}{l}
  x_1x_4 + x_2x_5 + x_3x_6,\\
x_1x_7 + x_2x_8 + x_3x_9.
\end{array}
\]
Let us continue analyzing the 20 quadrics saved in the variable \lstinline{f}. We use the following command in \texttt{Bertini} to determine whether our variety is reducible and compute its degree:
\begin{lstlisting}
import Bertini: bertini
bertini(round.(f), TrackType = 1, bertini_path = p1)
\end{lstlisting}
Here \lstinline{p1} is the path to the \texttt{Bertini} binary. \texttt{Bertini} confirms that the variety is irreducible of degree $8$ and dimension $3$
(cf.~Figure \ref{DDSO3}).

Using \texttt{Eirene} we construct the barcodes depicted in Figure~\ref{BarcodeSO3}. We run the following commands to plot barcodes for a random subsample of $250 $  points in ${\rm SO}(3)$:
\begin{lstlisting}
# sample 250 random points
i = rand(1:887, 250)
# compute the scaled Euclidean distances
dists = ScaledEuclidean(data[:,i])
# pass distance matrix to Eirene and plot barcodes in dimensions up to 3
C = eirene(dists, maxdim = 3)
barcode_plot(C, [0,1,2,3], [8,8,8,8])
\end{lstlisting}
The first array \lstinline{[0,1,2,3]} of the \lstinline{barcode_plot()} function specifies
the desired dimensions. The second array \lstinline{[8,8,8,8]} selects the 8 largest barcodes for each dimension. If the user does not pass the last array to the function, then all the barcodes are plotted.
 To compute barcodes arising from the complex specified in~(\ref{ellipsoids_relaxation}), we type
\begin{lstlisting}
dists = EllipsoidDistances(data[:,i], f, 1e-5)
C = eirene(dists, maxdim = 3)
barcode_plot(C, [0,1,2,3], [8,8,8,8])
\end{lstlisting}
Here, \lstinline{f = FindEquations(data, :with_qr, 2, false)} is the vector of $20$ quadrics. The third
argument of \lstinline{EllipsoidDistances}
is the parameter $\lambda$ from  (\ref{ellipsoids_relaxation}).
It is here set to $10^{-5}$.

\begin{figure}[h]
\begin{center}
\includegraphics[width=0.47\textwidth]{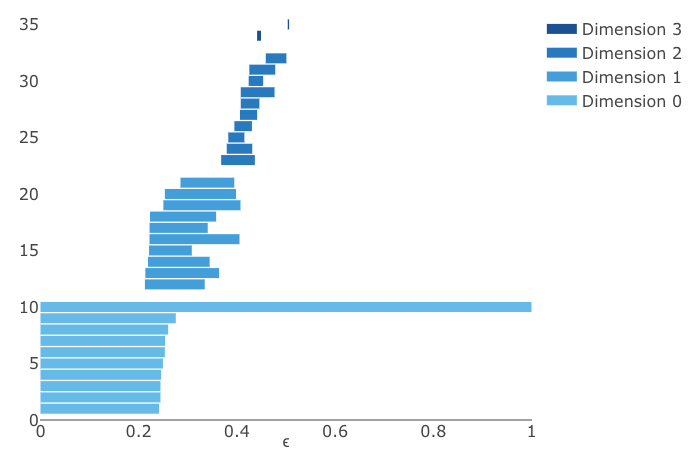} \quad
\includegraphics[width=0.47\textwidth]{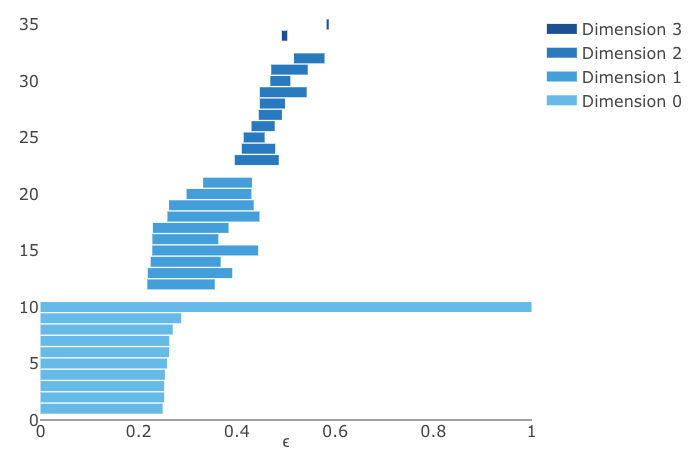}
\caption{Barcodes for a subsample of  $250$ points from $ \mathrm{SO}(3)$.
The left picture shows the standard Vietoris-Rips complex, while that on the right comes from the ellipsoid-driven complex (\ref{ellipsoids_relaxation}). Neither reveals any structures in dimension 3, though
$V = \mathrm{SO}(3)$ is diffeomorphic to $\PP^{3}_\RR$ and has a non-vanishing  $H_3(V, \ZZ)$.
}\label{BarcodeSO3}
\end{center}
\end{figure}

Our subsample of $250$ points is not dense enough to reveal features except in dimension~$0$.
Instead of randomly selecting the points in the subsample, one could also use the \emph{sequential maxmin landmark
selector}~\cite[\S 5.2]{JavaplexTutorial}. Subsamples chosen this way tend to cover the dataset and
to be spread apart from each other. One might also improve the result by constructing
different complexes, for example, the lazy witness
complexes in~\cite[\S 5]{JavaplexTutorial}. However, this is not implemented in \texttt{Eirene} at present.

\subsection{Dataset 2: a sample from the variety of rank one $2\times3$-matrices}

The second sample consists of $200$ data points from the
Segre variety $\PP^1_\RR \times \PP^2_\RR$ in $\PP^5_\RR$,
that is Example \ref{ex:rankone} with $m=n=3,\, r=1$.
 We load our sample into the {\tt Julia} session by typing
\begin{lstlisting}
data = datasets["2x3 rank one matrices"]
\end{lstlisting}
We try the \lstinline{DimensionDiagrams} command once with the Boolean value set to \lstinline{false} (Euclidean space) and once with the value set to \lstinline{true} (projective space).
The diagrams are depicted in Figure~\ref{DD2times3}.
 As the variety $V$ naturally lives in $\PP^5_\RR$,
 the projective diagrams yield better estimates and hint that the dimension is either $3$ or $4$. The true  dimension in $\PP^5_\RR$ is $3$.

\begin{figure}[h]
\begin{center}
\includegraphics[scale=0.34]{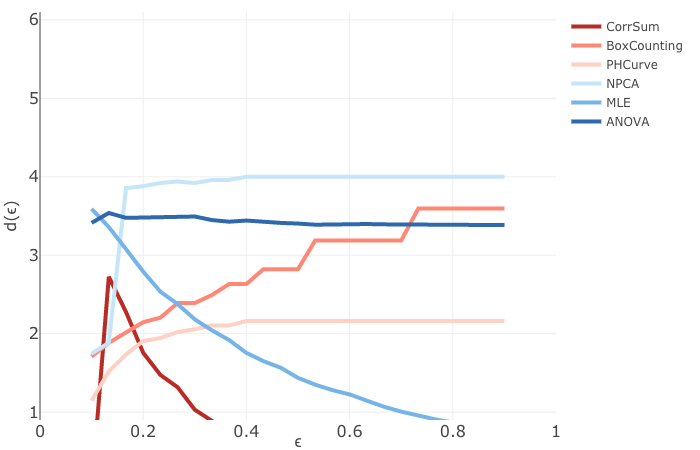}\includegraphics[scale=0.34] {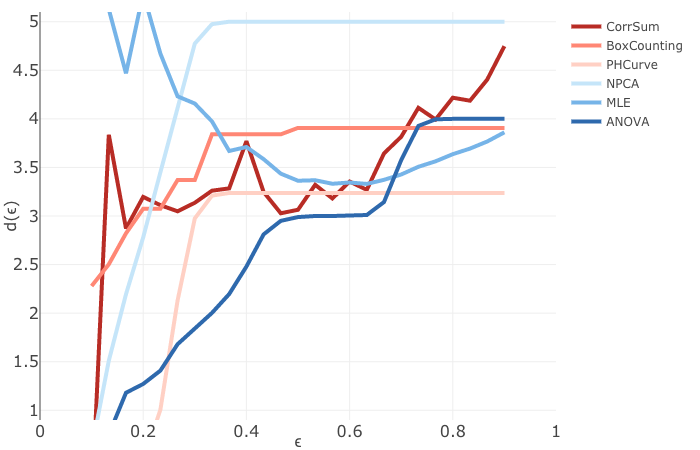}
\caption{Dimension diagrams for 200 points on the variety of $2 \times 3$
 matrices of rank $1$.
 The left picture shows dimension diagrams for the estimates in $\RR^{6}$.
 The right picture shows those for projective space $\PP^{5}_\RR$. }\label{DD2times3}
\end{center}
\end{figure}

The next step is to find polynomials that vanish.
We set \texttt{homogeneous\textunderscore equations} to \lstinline{true} and $d=2$: \lstinline{f = FindEquations(data, method, 2, true)}.
All three methods, SVD, QR and RREF, correctly report
the existence of three quadrics.
The equations obtained with QR after rounding are as desired:
$$
x_1x_4 - x_2x_3=0,\quad
x_1x_6 - x_2x_5=0,\quad
x_3x_6 - x_4x_5=0.$$
Running \texttt{Bertini} we verify that $V$
   is an irreducible variety of dimension $3$ and degree $3$.

We next estimate the volume of $ V $
using the formula in Theorem~\ref{kinematic}.
We intersect $V$ with $500$ random planes in $\PP^5_\RR$
  and count the number of real intersection points. We must initialize $500$ linear functions with Gaussian entries involving the same variables as \lstinline{f}:
\begin{lstlisting}
import MultivariatePolynomials: variables
X = variables(f)
Ls = [randn(3, 6) * X for i in 1:500]
\end{lstlisting}
Now, we compute the real intersection points using \texttt{HomotopyContinuation.jl}.
\begin{lstlisting}
using HomotopyContinuation
r = map(Ls) do L
  # we multiply with a random matrix to make the system square
  S = solve([randn(2,3) * f; L])
  # check which are solutions to f and return the real ones
  vals = [[fi(X => s) for fi in f] for s in solutions(S)]
  i = find(norm.(vals) .< 1e-10)
  return length(real(S[i]))
end
\end{lstlisting}
The command \lstinline{pi^2 * mean(r)} reports an estimate of $19.8181$ for the volume of $V$.
The true volume of $V$ is the
length of $ \PP^1_\RR $ times
the area of $ \PP^2_\RR$, which is
$\pi \cdot (2 \pi) = 19.7392$.

\begin{figure}[!h]
\begin{center}
\includegraphics[width=0.47\textwidth]{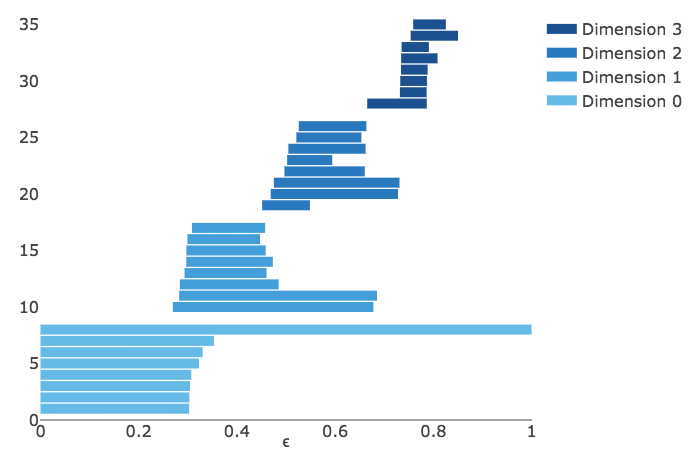} \quad
\includegraphics[width=0.47\textwidth]{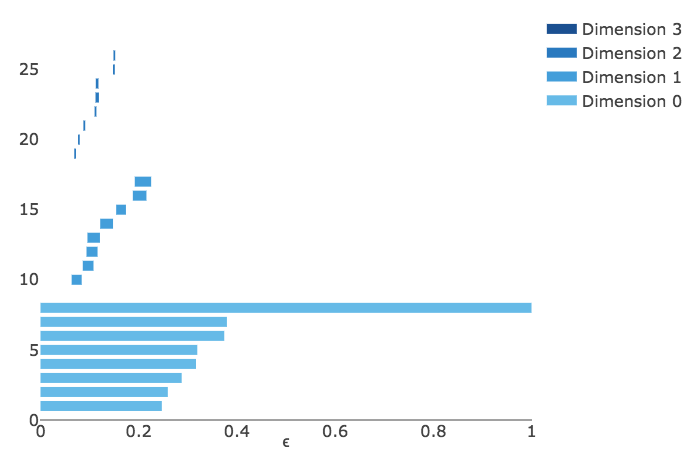}
\caption{Barcodes for $200$ points on the Segre variety of $2\times 3$  matrices of rank $1$.
The true mod $2$ Betti numbers of
     $\mathbb{P}^1_\mathbb{R} \times \mathbb{P}^2_\mathbb{R}$ are $1,2,2,1$.
  The left picture shows the barcodes for the usual Vietoris-Rips complex computed using scaled Fubini-Study distance. The right picture is computed using the scaled Euclidean distance.
             Using the Fubini-Study distance yields better results.
}\label{2x3 rank one matrices}
\end{center}
\end{figure}

Using \texttt{Eirene}, we construct the barcodes depicted in Figure~\ref{2x3 rank one matrices}. The barcodes constructed using Fubini-Study distance detect persistent features in dimensions $0$, $1$ and $2$.
The barcodes using Euclidean distance only have a strong topological signal in dimension $0$.

\subsection{Dataset 3: conformation space of cyclo-octane}\label{sub:cyclo-octane}

Our next variety $V$ is the  conformation space of the molecule cyclo-octane $C_8 H_{16}$.
We use the same sample $\Omega$ of $ 6040$ points that was analyzed in~\cite[\S.6.3]{JavaplexTutorial}.
Cyclo-octane consists of eight carbon atoms arranged in a ring and each bonded to a pair of hydrogen atoms (see Figure~\ref{cyclo_pic}).
The location of the hydrogen atoms is determined by that of the carbon atoms due to energy minimization. Hence, the conformation space of cyclo-octane consists
 of all possible spatial arrangements, up to rotation and translation, of the ring of carbon atoms.

\begin{figure}[h!]
\begin{center}
 \includegraphics[width = 0.23\textwidth]{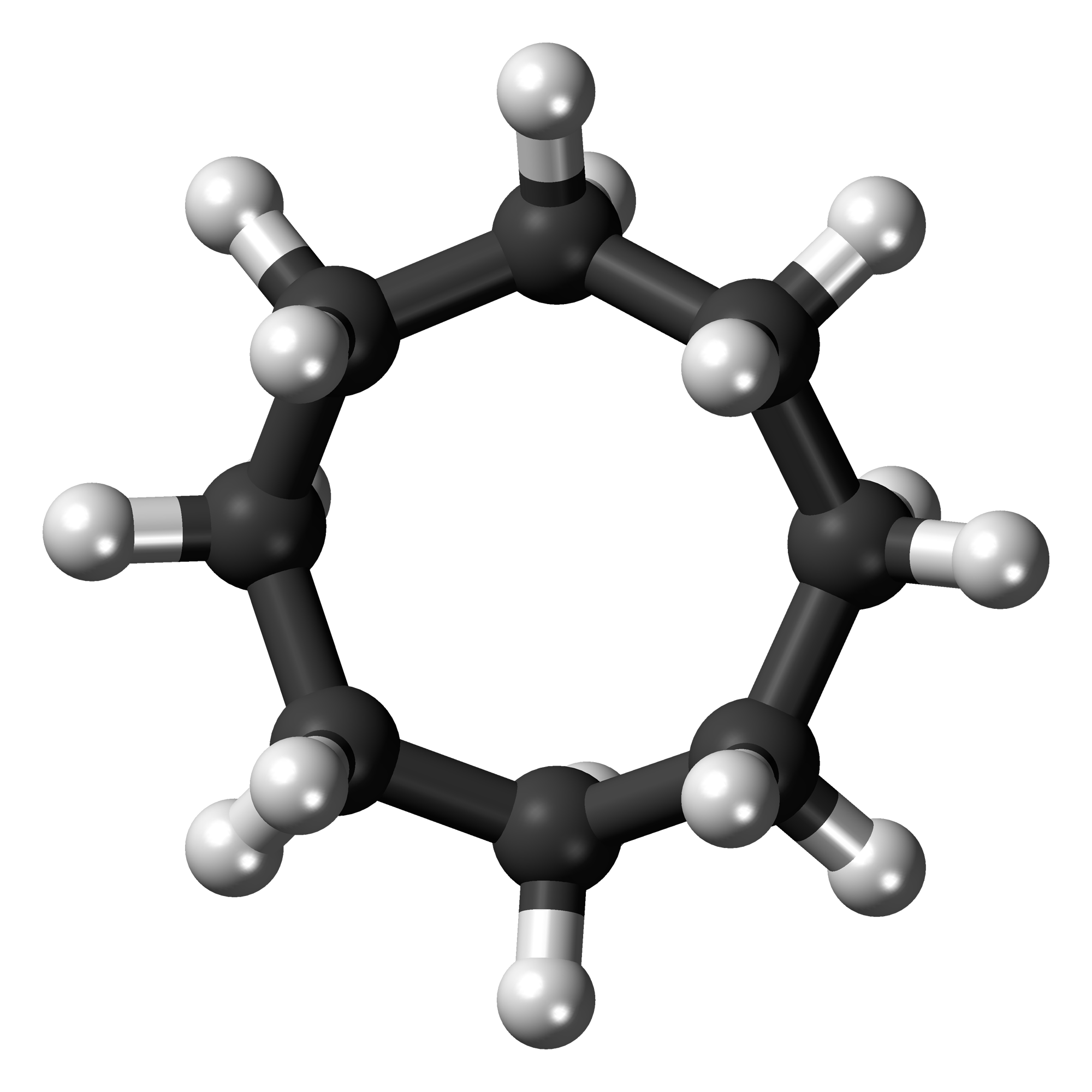}
 \caption{A cyclo-octane molecule. \label{cyclo_pic}}
\end{center}
\end{figure}

Each conformation is a point in $\mathbb{R}^{24} = \mathbb{R}^{8\cdot 3}$, which represents the coordinates of the carbon atoms $\{z_0, \dots, z_7\}\subset\mathbb{R}^3$. Every carbon atom $z_i$ forms an isosceles triangle with its two neighbors with angle $\frac{2\pi}{3}$ at $z_i$. By the law of cosines, there is a constant $c>0$ such that the squared distances $\,d_{i,j}=\Vert z_i-z_j\Vert^2\,$  satisfy
\begin{equation}
\label{eq:16quadrics}
d_{i,i+1}\,=\,c \quad \hbox{and} \quad d_{i,i+2}\,=\,\frac{8}{3}c
\quad \hbox{for all $i  \,\, $ (mod~$8$).}
\end{equation}
 Thus we expect to find $16$ quadrics from the given data. In our sample we have $c\approx 2.21$.

 The conformation space is defined modulo translations and rotation; i.e., modulo the $6$-dimensional
 {\em group  of rigid motions} in $\mathbb{R}^3$. An implicit representation of this
 quotient space arises by substituting (\ref{eq:16quadrics}) into the
 Sch\"onberg matrix of Example \ref{ex:8schoenberg}  with $p=8$ and $r=3$.

 However, the given $\Omega$ lives in  $\mathbb{R}^{24} = \mathbb{R}^{8\cdot 3}$,
 {\em i.e.~}it uses the coordinates of the carbon atoms.
  Since the group has dimension $6$, we expect to
 find~$6$ equations that encode a \emph{normal form}.
 That normal form is a distinguished representative from each orbit  of the group action.

\begin{figure}[h!]
\begin{center}
\includegraphics[scale=0.45]{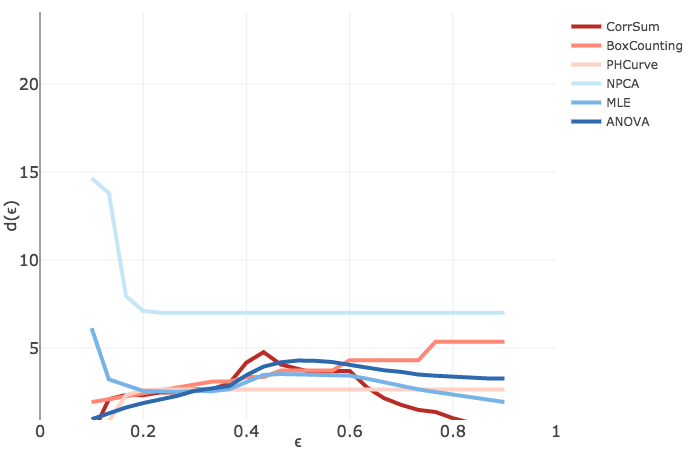}
\caption{Dimension diagrams for $420$ points from the cyclo-octane dataset.\label{dim_cyclo_plot}}
\end{center}
\end{figure}

Brown \emph{et al.} \cite{Brown} and Martin \emph{et al.}~\cite{Martin} show that the conformation space of cyclo-octane is the union of a sphere with a Klein bottle, glued together along two circles of singularities. Hence,
the dimension of $V$ is $2$, and it has Betti numbers $1,1, 2$ in mod $2$ coefficients.

To accelerate the computation of  dimension diagrams, we took a random subsample of
$ 420$ points. The output is displayed in Figure \ref{dim_cyclo_plot}.
A dimension estimate of $2$ seems reasonable:

\begin{lstlisting}
i = rand(1:6040, 420)
DimensionDiagrams(data[:,i], false)
\end{lstlisting}

The dataset $\Omega$ is noisy: each point is rounded to $4$ digits.
Direct use of \lstinline{FindEquations()} yields no polynomials vanishing on $\Omega$.
The reason is that our code sets the tolerance with
  the numerical rank in (\ref{eq:numericalrank}).
For noisy samples, we must set the tolerance manually. To get a sense for
adequate tolerance values, we first compute the multivariate Vandermonde matrix $U_{\leq d}(\Omega)$ and then plot the       base 10 logarithms of its singular values. We start with $d=1$.
\begin{lstlisting}
import PlotlyJS
M = MultivariateVandermondeMatrix(data, 1, false)
s = log10.(svdvals(M.Vandermonde))
p = PlotlyJS.scatter(; y=s, mode="lines", line_width = 4)
PlotlyJS.Plot(p)
\end{lstlisting}

This code produces the left plot in Figure \ref{svdvals}. This graph shows a clear drop
from $-0.2$ to $-2.5$. Picking the in-between value $-1$, we set the
 tolerance at $\tau = 10^{-1}$. Then, we type
\begin{lstlisting}
f = FindEquations(M, method, 1e-1)
\end{lstlisting}
where \lstinline{method} is one of our three methods. For this tolerance value we find six linear equations.
Computed using \texttt{:with\textunderscore qr} and rounded to three digits, they are as follows:
{\footnotesize
\begin{align*}
\mathrm{1.}\quad&-1.2x_1 - 3.5x_2 + 1.2x_3 - 4.2x_4 - 4.1x_5 + 3.9x_6 - 5.4x_7 - 2.0x_8 + 4.9x_9 - 5.4x_{10} + 2.2x_{11} + 4.9x_{12}\\
&\hspace{2cm}- 4.2x_{13} + 4.3x_{14} + 3.8x_{15} - 1.1x_{16} + 3.6x_{17} + x_{18}\\[0.2cm]
\mathrm{2.}\quad&-0.6x_1 - 1.3x_2 - 2.0x_4 - 1.3x_5 - 2.5x_7 - 2.5x_{10} + x_{11} - 2.0x_{13} + 2.4x_{14} - 0.5x_{16} + 2.3x_{17} + x_{20}\\[0.2cm]
\mathrm{3.}\quad&2.5x_1 + 8.1x_2 - 4.0x_3 + 9.2x_4 + 9.6x_5 - 10.5x_6 + 11.4x_7 + 4.7x_8 - 11.5x_9 + 12.6x_{10} - 5.1x_{11}\\&\hspace{2cm}  - 10.5x_{12} + 9.4x_{13} - 10.0x_{14} - 6.5x_{15} + 1.9x_{16} - 8.3x_{17} - 1.1x_{19} + x_{21}\\[0.2cm]
\mathrm{4.}\quad&x_1 + x_4 + x_7 + x_{10} + x_{13} + x_{16} + x_{19} + x_{22}\\[0.2cm]
\mathrm{5.}\quad&0.6x_1 + 2.3x_2 + 2.0x_4 + 2.3x_5 + 2.5x_7 + x_8 + 2.5x_{10} + 2.0x_{13} - 1.4x_{14} + 0.5x_{16} - 1.3x_{17} + x_{23}\\[0.2cm]
\mathrm{6.}\quad&-1.3x_1 - 4.6x_2 + 3.8x_3 - 4.9x_4 - 5.5x_5 + 7.5x_6 - 6.0x_7 - 2.7x_8 + 7.5x_9 - 7.2x_{10} + 2.9x_{11} + 6.5x_{12}\\
&\hspace{2cm} - 5.2x_{13} + 5.7x_{14} + 3.7x_{15} - 0.8x_{16} + 4.7x_{17} + 1.1x_{19} + x_{24}
\end{align*}}
We add the second and the fifth equation, and we add the first, third and sixth, by typing \lstinline{f[2]+f[5]} and \lstinline{f[1]+f[3]+f[6]} respectively. Together with \lstinline{f[1]}
  we get the following:
  \begin{equation}
  \label{eq:3by8}
  \begin{matrix}
&x_1 + x_4 + x_7 + x_{10} + x_{13} + x_{16} + x_{19} + x_{22}\\
&x_2 + x_5 + x_8 + x_{11} + x_{14} + x_{17} + x_{20} + x_{23}\\
&x_3 + x_6 + x_9 + x_{12} + x_{15} + x_{18} + x_{21} + x_{24}
\end{matrix}
\end{equation}
We learned that centering is the normal form for translation. We also learned
that the columns in (\ref{eq:3by8}) represent the eight atoms.
 Since we found $6$ linear equations, we believe that the three~$3$ remaining equations determine the normal form for rotations. However, we do not yet understand how the three degrees of rotation produce three linear constraints.

\begin{figure}[h!]
 \begin{center}
\includegraphics[width = 0.47\textwidth]{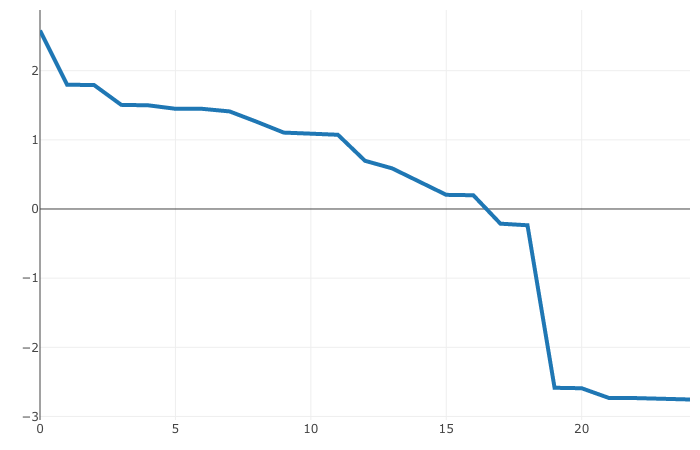}\quad \includegraphics[width = 0.47\textwidth]{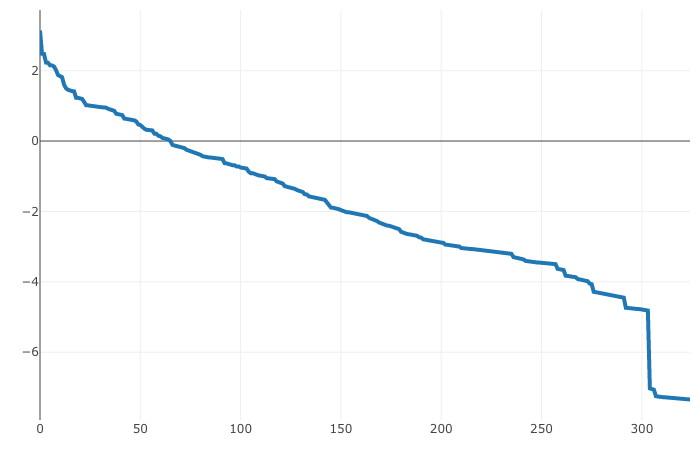}
\caption{Logarithms (base 10) of the singular values of the matrices $U_{\leq 1}(\Omega)$ (left)
and $U_{\leq 2}(\Omega)$ (right).\label{svdvals}}
\end{center}
\end{figure}

We next proceed to equations of degree $2$.
Our hope is to find the $16$ quadrics in (\ref{eq:16quadrics}).
 Let us check whether this works.
   Figure \ref{svdvals} on the right shows the logarithms of the singular values of
     the multivariate Vandermonde matrix $U_{\leq 2}(\Omega)$. Based  on this we set
 $\tau = 10^{-6}$.

The command \lstinline{FindEquations(M, :with_svd, 2, 1e-6)} reveals~$21 $ quadrics.
 However,  these are the pairwise products of the $6$ linear equations we found earlier.
    An explanation for why we cannot find the $16$ distance quadrics is as follows. Each of the $6$ linear equations evaluated at the points in $\Omega$ gives about $10^{-3}$ in our numerical computations. Thus their products equal about $10^{-6}$. The distance quadrics equal about $10^{-3}$.
  At tolerance $10^{-6}$, we miss them. Their values are much larger than the $10^{-6}$
  from the $21$ redundant quadrics. By randomly rotating and translating each data point,
 we can manipulate the dataset such that \lstinline{FindEquations} together with a tolerance value $\tau = 10^{-1}$ gives the $16$ desired quadrics. The fact that no linear equation vanishes on the manipulated dataset provides more evidence that 3 linear equations are determining the normal form for rotations.

\begin{figure}[h!]
 \begin{center}
\includegraphics[width = 0.48\textwidth]{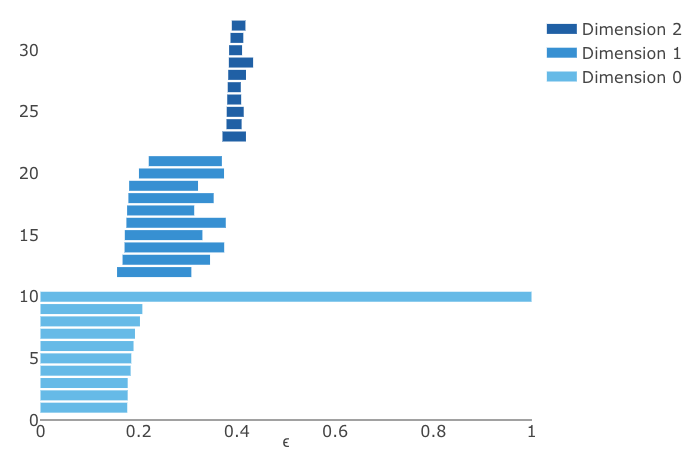} \quad
\includegraphics[width = 0.48\textwidth]{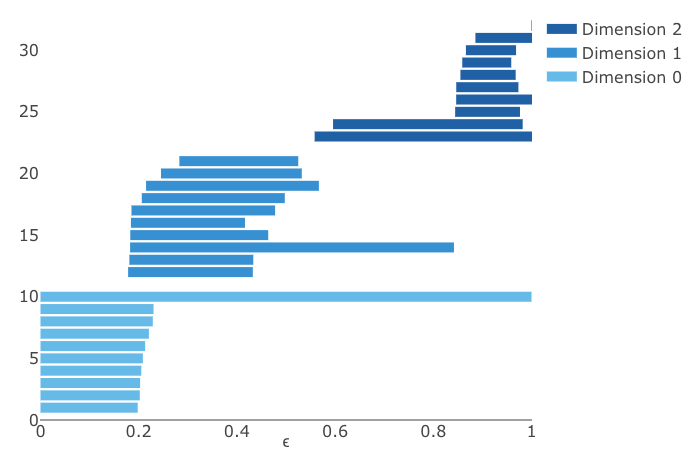}
\caption{Barcodes for a subsample of 500 points from the cyclo-octane dataset. The left plot shows the barcodes for the usual Vietoris-Rips complex. The right picture shows barcodes for the
ellipsoid-driven simplicial complex in
(\ref{ellipsoids_relaxation}). The right barcode correctly captures the homology
of the conformation space. \label{barcode_cyclo}}
\end{center}
\end{figure}

The cyclo-octane dataset was used in~\cite[\S.6.3]{JavaplexTutorial} to demonstrate
that persistent homology can efficiently recover the homology groups of the conformation space.
We confirmed this result using our software.
 We determined the barcodes for a random subsample of 500 points. In addition to
 computing with Vietoris-Rips complexes, we use the 6 linear equations and the 16 distance quadrics to produce the ellipsoid-driven barcode plots. The results are displayed in Figure \ref{barcode_cyclo}. The barcodes  from the usual Vietoris-Rips complex do not capture the correct homology groups, whereas the barcodes arising from our new complex (\ref{ellipsoids_relaxation}) do.

 \bigskip
 \bigskip

\noindent
{\bf Acknowledgements.} We thank Henry Adams, Mateo D\'iaz,
Jon Hauenstein, Peter Hintz, Ezra Miller, Steve Oudot, Benjamin Schweinhart,
 Elchanan Solomon, and Mauricio Velasco for helpful discussions.
 Bernd Sturmfels and Madeleine Weinstein acknowledge support from the
 US National Science Foundation.

 \bigskip
 \bigskip
 \bigskip

\footnotesize{

}

\bigskip
\bigskip

\noindent
\footnotesize {\bf Authors' addresses:}

\smallskip

\noindent Paul Breiding,  \ MPI-MiS Leipzig
\hfill {\tt Paul.Breiding@mis.mpg.de}

\noindent
Sara Kali\v snik,  \ MPI-MiS Leipzig and Wesleyan University
\hfill {\tt skalisnikver@wesleyan.edu}

\noindent Bernd Sturmfels,
 \  MPI-MiS Leipzig and
UC  Berkeley \hfill  {\tt bernd@mis.mpg.de}

\noindent Madeleine Weinstein,
UC  Berkeley \hfill {\tt maddie@math.berkeley.edu}

\end{document}